\documentclass[A4paper,11pt]{article}
\usepackage{amsmath,amssymb,amsfonts,amsthm}
\usepackage{enumerate}
\usepackage{paralist}
\usepackage{graphics} 
\usepackage{float}
\usepackage{epsfig} 
\usepackage{graphicx}
\usepackage{epstopdf}
\usepackage{hyperref}

\usepackage{bm}
\usepackage{color}

\usepackage[a4paper]{geometry}

\newtheorem{theorem}{Theorem}[section]

\newtheorem{lemma}[theorem]{Lemma}
\newtheorem{proposition}[theorem]{Proposition}

\theoremstyle{definition}
\newtheorem{definition}[theorem]{Definition}
\newtheorem{remark}{Remark}

\newcommand{\Lip}{\textup{Lip}}
\newcommand{\loc}{\textup{loc}}

\newcommand{\N}{\mathbb{N}}

\newcommand{\R}{\mathbb{R}}

\newcommand{\cl}{\mathcal{L}}

\newcommand{\ct}{\mathcal{T}}
\newcommand{\W}{\mathcal{W}}
\newcommand{\PP}{\mathcal{P}_1}
\newcommand{\PC}{\mathcal{P}_c}
\DeclareMathOperator{\supp}{supp}

\DeclareMathOperator{\argmin}{arg\,min}

\DeclareMathOperator*{\esssup}{ess\,sup}

\newcommand{\Fun}[1]{F^{[#1]}}
\newcommand{\Energy}{\mathcal E}
\newcommand{\Ea}[1]{\Energy^{[a]}(#1)}
\newcommand{\Eahat}{\Ea{\widehat a}}
\newcommand{\Ean}{\Energy^{[a],N}}
\newcommand{\Eahatn}{\Ean(\widehat a)}
\newcommand{\prerho}{\overline\rho}
\newcommand{\x}{x^{[a]}}
\newcommand{\xahat}{x^{[\widehat a]}}
\newcommand{\dotx}{\dot{x}^{[a]}}
\newcommand{\dotxahat}{\dot{x}^{[\widehat a]}}

\allowdisplaybreaks

\title{Inferring Interaction Rules from Observations of Evolutive Systems I: The Variational Approach}

\author{M. Bongini\footnote{Faculty of Mathematics, Technical University of Munich, Boltzmannstrasse 3, 85748 Garching bei M\"unchen, Germany, Email: mattia.bongini@ma.tum.de}, 
M. Fornasie\footnote{Faculty of Mathematics, Technical University of Munich, Boltzmannstrasse 3, 85748 Garching bei M\"unchen, Germany, Email: massimo.fornasier@ma.tum.de}r, 
M. Hansen\footnote{Faculty of Mathematics, Technical University of Munich, Boltzmannstrasse 3, 85748 Garching bei M\"unchen, Germany, Email: markus.hansen@ma.tum.de}, 
and M. Maggioni\footnote{Department of Mathematics, Duke University, 	117 Physics Bldg., Science Dr., Box 90320
Durham, NC 27708-0320
U.S.A., Email: mauro@math.duke.edu}}

\date{}

\begin{document}
\maketitle

\begin{abstract}
In this paper we are concerned with the  learnability of nonlocal interaction kernels for  first order systems modeling certain social interactions, from observations of realizations of their dynamics. This paper is the first  of a series  on learnability of nonlocal interaction kernels and presents a variational approach to the problem. In particular, we assume here that  the kernel to be learned is bounded and locally Lipschitz continuous and that the initial conditions of the systems are drawn identically and independently at random according to a given initial probability distribution. Then the minimization over a rather arbitrary  sequence of (finite dimensional) subspaces of a least square functional measuring the discrepancy from observed trajectories  produces uniform approximations to the kernel on compact sets. The convergence result is obtained by combining mean-field limits, transport methods, and a $\Gamma$-convergence argument. A crucial condition for the learnability is a certain coercivity property of the least square functional, majoring an $L_2$-norm discrepancy to the kernel with respect to a probability measure, depending on the given initial probability distribution by suitable push forwards and transport maps. We illustrate the convergence result by means of several numerical experiments. 
\end{abstract}
{\bf Keywords}: nonlocal interaction kernel learning, first order nonlocal interaction equations, mean-field equations, $\Gamma$-convergence

\bigskip

\tableofcontents

\section{Introduction}

What are the instinctive individual reactions which make a group of animals forming coordinated movements, for instance a flock of migrating birds or a school of fish? Which biological interactions between cells produce the formation of complex structures, like tissues and organs? What are the mechanisms which induce certain significant changes in a large amount of players in the financial market? 
In this paper we are concerned with the ``mathematization'' of the problem of learning or inferring interaction rules from observations of evolutions. The framework we consider is the one of  evolutions driven by gradient descents.
The study of gradient flow evolutions to minimize  certain energetic landscapes has been the subject of intensive research in the past years \cite{AGS}. Some of the most recent models are  aiming at describing time-dependent phenomena also in biology or even in social dynamics, borrowing a leaf from more established and classical  models in physics. For instance, starting with the seminal papers of Vicsek et. al. \cite{VCBCS95} and Cucker-Smale \cite{CucSma07}, there has been a flood of models describing consensus or opinion formation,  modeling the exchange of information as long-range social interactions (forces) between active agents (particles). However, for the analysis, but even more crucially for the reliable and realistic numerical simulation of such phenomena, one presupposes a complete understanding and determination of the governing energies. Unfortunately, except for physical situations where the calibration of the model can be done by measuring the governing forces rather precisely, for some relevant macroscopical models in physics and most of the models in biology and social sciences the governing energies are far from being precisely determined. In fact, very often in these studies the governing energies are just predetermined to be able to reproduce, at least approximately or qualitatively, some of the macroscopical effects of the observed dynamics, such as the formation of certain patterns, but there has been relatively little effort in the applied mathematics literature towards matching data from real-life cases. 

In this paper we aim at bridging in the specific setting of first order models, the well-developed theory of dynamical systems and mean-field equations with classical  approaches of approximation theory, nonlinear time series analysis, and machine learning. We  provide a  mathematical framework for the reliable identification of the governing forces from data obtained by direct observations of corresponding time-dependent evolutions. This is a new kind of inverse problem, beyond more traditionally considered ones, as the forward map is a strongly nonlinear  evolution, highly dependent on the probability measure generating the initial conditions. As we aim at a precise quantitative analysis, and to be very concrete, we will  attack the learning of the governing laws of evolution for specific models in social dynamics governed by nonlocal interactions. The models considered in the scope of this paper are deterministic, however we intend in follow up work to extend our results towards stochastic dynamical systems.


\subsection{General abstract framework}

Many time-dependent phenomena in physics, biology, and social sciences can be modeled by a function $x:[0,T] \to \mathcal H$, where $\mathcal H$ represents the space of states of the physical, biological or social system, which evolves from an initial configuration $x(0)=x_0$  towards a more convenient state or a new equilibrium. The space $\mathcal H$ can be a conveniently chosen Banach space or just a metric space; let $\operatorname{dist}_{\mathcal H}$ be the metric on $\mathcal H$.
This implicitly assumes that $x$ evolves driven by a minimization process of a potential energy $\mathcal J: \mathcal H \times [0,T] \to \mathbb R$.  In this preliminary introduction we consciously avoid specific assumptions on  $\mathcal J$, as we wish to keep a rather general view. We restrict the presentation to particular cases below.%

Inspired by physics, for which conservative forces are the derivatives of the potential energies, one can describe the evolution as satisfying a gradient flow inclusion of the type
\begin{equation}\label{gradientflow}
\dot x(t) \in - \partial_x \mathcal J(x(t),t),
\end{equation}
where $\partial_x \mathcal J(x,t)$ is some notion of differential of $\mathcal J$ with respect to $x$, which might already take into consideration additional constraints which are binding the states to certain sets.

\subsection{Example of gradient flow of nonlocally interacting particles}\label{sec:gradflow}
Let us introduce an example of the general framework described above. It is actually the main focus of this paper.
Assume that $x=(x_1,\dots,x_N) \in \mathcal H \equiv  \mathbb R^{d\times N}$ and that 
$$
\mathcal J_N(x) = \frac{1}{2N} \sum_{i,j=1}^N A(| x_i -  x_j |),
$$
where $A:\mathbb R_+ \to \mathbb R$ is a suitable nonlinear interaction kernel function, which, for simplicity we assume to be smooth (see below more precise conditions), and $|\cdot|$ is the Euclidean norm in $\mathbb R^d$. Then, the formal unconstrained gradient flow \eqref{gradientflow} associated to this energy is written coordinatewise as
\begin{equation}\label{fdgradientflow}
\dot x_i(t) = \frac{1}{N} \sum_{j \neq i} \frac{A'(| x_i(t) -  x_j(t) |)}{| x_i(t) -  x_j(t) |} (x_j(t) - x_i(t)), \quad i=1,\dots,N.
\end{equation}
Under suitable assumptions of local Lipschitz continuity and boundedness of the interaction function
\begin{equation}\label{intker}
a(\cdot) := \frac{A'(|\cdot|)}{| \cdot |},
\end{equation} this evolution is well-posed for any given $x(0)=x_0$ and it is expected to converge for $t \to \infty$ to configurations of the points whose mutual distances are close to local minimizers of the function $A$, representing steady states of the evolution as well as critical points of $\mathcal J_N$.\\
It is also well-known, see \cite{AGS} and Proposition \ref{pr:exist} below, that for $N \to \infty$ a mean-field approximation holds: if the initial conditions $x_i(0)$ are i.i.d. according to a compactly supported probability measure $\mu_0 \in \mathcal P_c(\mathbb R^d)$ for $i=1,2,3, \dots$, the empirical measure $\mu^N(t) = \frac{1}{N} \sum_{i=1}^N \delta_{x_i(t)}$ weakly converges for $N \to \infty$  to the probability measure-valued trajectory $t \mapsto \mu(t)$ satisfying  the equation
\begin{equation}\label{eq:meanfield}
\partial_t \mu(t) = - \nabla \cdot ((\Fun{a} * \mu(t)) \mu(t)), \quad \mu(0)=\mu_0,
\end{equation}
in weak sense, 
where $\Fun{a}(z) =-a(|z|)z=-A'(|z|)z/|z|$, for $z \in \mathbb R^{d}$. In fact the differential equation \eqref{eq:meanfield} corresponds again to a gradient flow of the ``energy''
$$
\mathcal J (\mu) = \int_{\mathbb R^{d\times d}} A(| x-  y |) d \mu(x) d\mu(y),
$$
on the metric space $\mathcal H =\mathcal P_c(\mathbb R^d)$ endowed with the so-called Wasserstein distance. Continuity equations of the type \eqref{eq:meanfield} with nonlocal interaction kernels are currently the subject of intensive research  towards the modeling of the biological and social behavior of microorganisms, animals, humans, etc. We refer to the  articles \cite{13-Carrillo-Choi-Hauray-MFL,cafotove10} for recent overviews on this subject. Despite the tremendous theoretical success of such research direction in terms of mathematical results on well-posedness and asymptotic behavior of solutions, as we shall stress below in more detail, one of the issues which is so far scarcely addressed in the study of models of the type \eqref{fdgradientflow} or \eqref{eq:meanfield} is their actual applicability. Most of the results are addressing a purely {\it qualitative analysis} given certain smoothness and asymptotic properties of the kernels $A$ or $a$ at the origin or at infinity, in terms of well-posedness or in terms of asymptotic behavior of the solution for $t \to \infty$.  Certainly such results are of great importance, as such interaction functions, if ever they can really describe social dynamics,  are likely to differ significantly from well-known models from physics and it is reasonable and legitimate to consider a large variety of classes of such functions.
However, a solid mathematical framework which establishes the conditions of ``learnability'' of the interaction kernels from observations of the dynamics is currently not available and it will be the main subject of this paper.

\subsection{Parametric energies and their identifications}

Let us now return to consider again an abstract energy $\mathcal J^{[a]}$ and let us assume that it is dependent on a parameter function $a$, as indicated in the superscript. As in the example mentioned above, $a$ may be defining a nonlocal interaction kernel as in  \eqref{intker}. The parameter function $a$ not only determines the abstract energy, but also the corresponding evolutions  $t \mapsto \x(t)$ driven according to \eqref{gradientflow}, for fixed initial conditions $\x(0)=x_0$. (Here we assume that the class of $a$ is such that the evolutions exist and they are essentially well-posed; we explicitly stress again the dependency on $a$ with a superscript $[a]$, which below we may remove as soon as such dependency is clear from the context.)
The fundamental question to be here addressed is: can we recover $a$ with high accuracy given some observations of the realized evolutions? This question is prone to several specifications, for instance, we may want to assume that the initial conditions are generated according to a certain probability distribution or they are chosen deterministically ad hoc to determine at best $a$, that the observations are complete or incomplete, etc. As one  quickly realizes, this is a very broad field to explore with many possible developments. Surprisingly, there are no results in this direction at this level of generality, and relatively little is done in the specific directions we mentioned in the example above. We  refer, for instance,  to \cite{parisi08,parisi08-1,parisi08-2,Hildenbrandt01112010,mann11,heoemascszwa11} and references therein, for groundbreaking statistical studies on the inference of social rules in collective behavior. 
\subsection{The optimal control approach and its drawbacks}
Let us introduce an approach, which  perhaps would be  naturally  considered at a first instance, and focus for a moment on the gradient flow model \eqref{gradientflow}. Given a certain gradient flow evolution $t \mapsto \x(t)$ depending on the unknown parameter function $a$, one might decide to design the recovery of $a$ as an optimal control problem \cite{brpi07}: for instance, we may seek a parameter function $\widehat a$ which minimizes
\begin{equation}\label{optcontr}
\Eahat = \frac{1}{T}\int_0^T \left [ \operatorname{dist}_{\mathcal H}(\x(s) - \xahat(s))^2 + \mathcal R(\widehat a) \right ] d s ,
\end{equation}
being $t \mapsto \xahat(t)$ the solution of gradient flow \eqref{gradientflow} for $\mathcal J = \mathcal J^{[\widehat a]}$, i.e.,
\begin{equation}\label{gradientflow2}
\dotxahat(t) \in - \partial_x \mathcal J^{[a]}(x^{[\widehat a]}(t),t),
\end{equation}
and $\mathcal R(\cdot)$ is a suitable regularization functional, which restricts the possible minimizers of \eqref{optcontr} to a specific class. The first fundamental problem one immediately encounters with this formulation is the strongly nonlinear dependency of $t \mapsto \xahat(t)$ on $\widehat a$, which results in a strong non-convexity of the functional \eqref{optcontr}. This also implies that a direct minimization of \eqref{optcontr} would risk to lead to suboptimal solutions, and even the computation of a first order optimality condition in terms of Pontryagin's minimum principle would not characterize uniquely the minimal solutions. Besides these fundamental hurdles, the numerical implementation of either strategy (direct optimization or solution of the first order optimality conditions) is expected to be computationally unfeasible to reasonable degree of accuracy as soon as the number of particles $N$ is significantly large (the well-known term {\it curse of dimensionality} conied by Richard E. Bellman for optimal control problems).

\subsection{A variational approach towards learning parameter functions in nonlocal energies}\label{sec:wp2}

Let us now consider again the more specific framework of the example in Section \ref{sec:gradflow}. We restrict our attention to interaction kernels $a$ belonging to the following \textit{set of admissible kernels}
\begin{align*}
	X=\bigl\{b:\R_+\rightarrow\R\,|\ b \in L_{\infty}(\R_+) \cap W^{1}_{\infty,\loc}(\R_+) \bigr \}.
\end{align*}
In particular every $a \in X$ is weakly differentiable, and its local Lipschitz constant $\Lip_{K}(a)$ is finite for every compact set $K \subset \R_+$.
Our goal is to learn the unknown interaction function $a \in X$ from the observation of the dynamics of the empirical measure $\mu^N$, defined by $\mu^N(t)=\frac{1}{N} \sum_{i=1}^N \delta_{\x_i(t)}$, where $\x_i(t)$ are driven by the interaction kernel $a$ according to the equations 
\begin{equation}\label{fdgradientflow2}
\dotx_i(t) = \frac{1}{N} \sum_{j \neq i} a(| \x_i(t) -  \x_j(t) |) (\x_j(t) - \x_i(t)), \quad i=1,\dots,N.
\end{equation}
Instead of the nonconvex optimal control problem above, we propose an alternative, direct approach which is both computationally very efficient and guarantees accurate approximations under reasonable assumptions.
In particular, we consider as an estimator of the kernel $a$ a minimizer of the following \textit{discrete error functional}
\begin{align}\label{eq-def-error1}
	\begin{split}
	\Eahatn = \frac{1}{T}\int_0^T\frac{1}{N}\sum_{i=1}^N\biggl|\frac{1}{N}\sum_{j=1}^N
			\left(\widehat a(|\x_i(t)-\x_j(t)|)(\x_i(t) - \x_j(t))-\dotx_i(t)\right)\biggr|^2 dt,
	\end{split}
\end{align}
among all competitor functions $\widehat a \in X$. 
Actually, the minimization of $\mathcal E^{[a],N}$ has a close connection to the optimal control problem, as it also promotes the minimization of the discrepancy $ \operatorname{dist}_{\mathcal H}(\x(s) - \xahat(s))^2$ in  \eqref{optcontr} (here we remind that in this setting $\mathcal H$ is the Euclidean space $\mathbb R^{d \times N}$):
\begin{proposition}\label{trajapprox}
If $a,\widehat a \in X$ then there exist a constant $C>0$ depending on $T, a$, and $\x(0)$ such that
\begin{equation}\label{eq:trajapprox}
 \operatorname{dist}_{\mathcal H}(\x(s) - \xahat(s))^2= \| \x (t) -\xahat  (t) \|^2 \leq C {\Eahatn}, 
\end{equation}
for all $t \in [0,T]$, where $\x$ and  $\xahat  $ are the solutions of \eqref{fdgradientflow2} for the interaction kernels $a$ and $\widehat a$ respectively. (Here $\| x \|^2 = \frac{1}{N} \sum_{i=1}^N |x_i|^2$, for $x \in \mathbb R^{d \times N}$.)
\end{proposition}
Therefore, if $\widehat a$ makes $\Eahatn$ small, the trajectories $t\to \xahat(t)$ of  system \eqref{fdgradientflow2} with interaction kernel $\widehat a$ instead of $a$ are as well a good approximation of the trajectories $t \mapsto \x(t)$ at finite time. 
The proof of this statement follows by Jensen's inequality and an application of Gronwall's lemma, as reported in detail in Section \ref{sec:learn}. \\

 For simplicity of notations, we may choose to ignore below the dependence on $a$ of the trajectories,  and write $x \equiv \x$ when such a dependence is clear from the context. Additionally, whenever we consider the limit $N \to \infty$,  we may denote the dependency of the trajectory on the number of particles $N \in \mathbb N$ by setting $x^N\equiv x \equiv \x$.
\\

Contrary to the optimal control approach, the functional $\Ean$ is convex and can be easily computed from witnessed trajectories $x_i(t)$ and $\dot x_i(t)$. We may even consider discrete-time approximations of the time derivatives $\dot x_i$ (e.g., by finite differences) and we shall assume that the data of the problem is the full set of observations $x_i(t)$ for $t \in [0,T]$, for a prescribed finite time horizon $T>0$. Furthermore, being a simple quadratic functional, its minimizers can be efficiently numerically approximated on a finite element space: given a finite dimensional space $V \subset X$, we let
\begin{equation}\label{fdproxy}
\widehat a_{N,V} = \argmin_{\widehat a \in V} \Eahatn.
\end{equation}
The fundamental mathematical question addressed in this paper is
\begin{itemize}
\item[(Q)] For which choice of the approximating spaces $V \in \Lambda$ (we assume here that $\Lambda$ is a countable family of approximating subspaces of $X$) does $\widehat a_{N,V} \to a$ for $N \to \infty$ and $V \to X$ and in which topology should the convergence hold?
\end{itemize}
We show now how we address this issue in detail by a variational approach, seeking a limit functional, for which techniques of $\Gamma$-convergence \cite{MR1201152}, whose general aim is establishing the convergence of minimizers for a sequence of equi-coercive functionals to minimizers of a target functional, may provide a clear characterization of the limits for the sequence of minimizers $(\widehat a_{N,V})_{N \in \mathbb N, V \in \Lambda}$.
Recalling again that $\Fun{a}(z) =-a(| z |)z$, for $z \in \mathbb R^{d}$, we rewrite the functional \eqref{eq-def-error1} as follows:
\begin{align}\label{pirlo}
	\begin{split}
	\Eahatn & = \frac{1}{T}\int_0^T\frac{1}{N}\sum_{i=1}^N\biggl|\frac{1}{N}\sum_{j=1}^N
			\bigl(\Fun{\widehat a}-\Fun{a}\bigr)(x_i-x_j)\biggr|^2 dt\\
			& = \frac{1}{T}\int_0^T \int_{\R^d} \biggl|\bigl(\Fun{\widehat a}-\Fun{a}\bigr)\ast\mu^N(t)\biggr|^2d\mu^N(t)(x)dt,
	\end{split}
\end{align}
for $\mu^N(t) = \frac{1}{N}\sum^N_{i = 1} \delta_ {x_i(t)}$.
This formulation of the functional makes it easy to recognize that the candidate for a $\Gamma$-limit  is then
\begin{align}\label{ourfunctional}
	\Eahat=\frac{1}{T}\int_0^T \int_{\R^d} \biggl|\bigl(\Fun{\widehat a}-\Fun{a}\bigr)\ast\mu(t)\biggr|^2d\mu(t)(x)dt,
\end{align}
where $\mu$ is a weak solution to the mean-field equation \eqref{meanfield}, as soon as the initial conditions $x_i(0)$ are identically and independently 
distributed according to a compactly supported probability measure $\mu(0)=\mu_0$. 

Although all of this is very natural, several issues  need to be addressed at this point.
The first one is to establish the space where a result of $\Gamma$-convergence may hold and the identification of $a$ can take place.
As the trajectories $t\mapsto x(t)$ do not explore the whole  space in finite time, we expect that such a space may {\it not} be independent of the initial probability measure $\mu_0$, as we clarify immediately.
By Jensen inequality we have
\begin{eqnarray}
\Eahat  &\leq & \frac{1}{T}\int_0^T  \int_{\R^d} \int_{\R^d}  | \widehat a(| x-y|) - a(| x-y|)|^2 | x-y|^2d \mu(t)(x) d \mu(t)(y) dt  \label{est-functional-1}\\
&= & \frac{1}{T}  \int_0^T\int_{\R_+}\bigl|\widehat a(s)-a(s)\bigr|^2 s^2d\varrho(t)(s) dt \label{midpoint1}
\end{eqnarray}
where $\varrho(t)$ is the pushforward of $\mu(t)\otimes\mu(t)$ by the Euclidean distance map
	$d:\R^d\times\R^d\rightarrow\R_+$ 
 	defined by $(x,y)\mapsto d(x,y)=|x-y|$. 
In other words, $\varrho:[0,T]\rightarrow \mathcal{P}_1(\R_+)$ is defined for every Borel set $A\subset\R_+$ as $\varrho(t)(A)=(\mu(t)\otimes\mu(t))\bigl(d^{-1}(A)\bigr)$.
The mapping $t \in [0,T] \mapsto\varrho(t)(A)$ is lower semi-continuous for every open set $A\subseteq\R_+$, and it is upper semi-continuous  for any compact set $A$ (see Lemma \ref{rhosc}).
We may therefore define a time-averaged probability measure $\prerho$ on the Borel $\sigma$-algebra of $\R_+$ by averaging $\varrho(t)$ over $t \in [0,T]$: for any open set $A \subseteq \R_+$ we define
\begin{align}\label{finallyrho}
	\prerho(A):=\frac{1}{T}\int_0^T\varrho(t)(A)dt,
\end{align}
and extend this set function to a probability measure on all Borel sets. Finally we define
\begin{equation}
 \rho(A):=\int_A s^2 d\prerho(s),
 \label{rho}
\end{equation}
for any Borel set  $A\subseteq\R_+$, to take into account the polynomial weight $s^2$ as appearing in \eqref{midpoint1}.
Then one can reformulate \eqref{midpoint1} in a very compact form as follows
\begin{align}\label{midpoint2}
	\Eahat\leq \int_{\R_+}\bigl|\widehat a(s)-a(s)\bigr|^2 d \rho(s)
		=\| \widehat a - a \|^2_{L_2(\mathbb R_+,\rho)}.
\end{align}
Notice that $\rho$ is defined through $\mu(t)$ which depends on the initial probability measure $\mu_0$. 

To establish coercivity of the learning problem
it is essential  to assume that there exists  $c_T>0$ such that also the following additional lower bound holds
\begin{align}\label{eq-coercive}
	c_T \| \widehat a - a \|^2_{L_2(\mathbb R_+,\rho)} \leq \Eahat,
\end{align}
for all relevant $\widehat a \in X \cap  L_2(\mathbb R_+,\rho)$. This crucial assumption eventually determines also the natural space $X \cap  L_2(\mathbb R_+,\rho)$ for the solutions,
which therefore depends on the choice of the initial conditions $\mu_0$. In particular the constant $c_T\geq 0$ might not be non-degenerate for all the choices of $\mu_0$
and one has to pick the initial distribution so that \eqref{eq-coercive} can hold for $c_T >0$. 
In Section \ref{sec:coerc} we show that for some specific choices of $a$ and rather general choices of $\widehat a \in X$ one can construct probability measure-valued trajectories $t \mapsto \mu(t)$ which allow to validate
\eqref{eq-coercive}.\\
In order to ensure compactness of the sequence of minimizers of $\mathcal E^{[a],N}$, we shall need to restrict the sets of possible solutions to classes of the type
\begin{align*}
X_{M,K} = \left\{b \in W^{1}_{\infty}(K) :
 \|b\|_{L_{\infty}(K)} + \|b'\|_{L_{\infty}(K)} \leq M
 \right\},
\end{align*}
where $M>0$ is some predetermined constant and $K \subset \mathbb R_+$ is a suitable compact set.


We now introduce the key property that a family of approximation spaces $V_N$ must possess in order to ensure that the minimizers of the functionals  $\mathcal E^{[a],N}$ over $V_N$ converge to minimizers of $\mathcal E^{[a]}$.

\begin{definition}\label{VNdef}
Let $M > 0$ and $K=[0,2R]$ interval in $\R_+$  be given. We say that a family of closed subsets $V_N \subset X_{M,K}$, $N \in \N$ has the \emph{uniform approximation property} in $L_{\infty}(K)$ if for all $b\in X_{M,K}$ there exists a sequence $(b_N)_{N \in \N}$ converging uniformly to $b$ on $K$ and such that $b_N\in V_N$ for every $N \in \N$.
\end{definition}

We are ready to state the main result of the paper:

\begin{theorem}\label{thm} Assume $a\in X$, fix $\mu_0 \in \mathcal{P}_c(\R^d)$ and let  $K=[0,2R]$ be an interval in $\R_+$ with $R>0$ as in Proposition \ref{pr:exist}.
	Set
	\begin{align*}
	M \geq \|a\|_{L_{\infty}(K)} + \|a'\|_{L_{\infty}(K)}.
	\end{align*}
	For every $N \in \N$, let $x_{0,1}^N,\ldots,x_{0,N}^N$ be i.i. $\mu_0$-distributed and define  $\mathcal E^{[a],N}$ as in \eqref{pirlo} for the solution $\mu^N$ of the equation \eqref{eq:meanfield} with initial datum
	\begin{align*}
	\mu^N_0 = \frac{1}{N}\sum^N_{i = 1} \delta_{x_{0,i}^N}.
	\end{align*}
	For  $N \in \N$, let $V_N\subset X_{M,K}$ be a sequence of subsets with the \emph{uniform approximation property} as in Definition \ref{VNdef} and consider
	\begin{align*}
		\widehat a_N\in\argmin_{\widehat a\in V_N} \Eahatn.
	\end{align*}
	
Then the sequence $(\widehat a_{N})_{N \in \N}$ has a subsequence converging uniformly on $K$ to some continuous function $\widehat a \in X_{M,K}$ such that
$\mathcal E^{[a]}(\widehat a)=0$. \\
 If we additionally assume the coercivity condition \eqref{eq-coercive}, then $\widehat a=a$ in $L_2(\R_+,\rho)$. Moreover, in this latter case, if there exist  rates $\alpha,\beta >0$, constants $C_1,C_2>0$, and a sequence $(a_N)_{N \in \mathbb N}$ of elements $a_N \in V_N$ such that 
\begin{equation}\label{rate1}
 \| a - a_N \|_{L_\infty(K)} \leq C_1 N^{-\alpha},
\end{equation}
and 
\begin{equation}\label{rate2}
 \mathcal W_1(\mu_0^N,\mu_0) \leq C_2 N^{-\beta},
\end{equation}
then there exists a constant $C_3>0$ such that 
\begin{equation}\label{rateapprox}
 \| a - \widehat a_N \|_{L_2(\R_+,\rho)}^2 \leq C_3 N^{-\min\{ \alpha, \beta\}},
\end{equation}
for all $N \in \mathbb N$. In particular, in this case, it is the entire sequence $(\widehat a_N)_{N \in \mathbb N}$ (and not only subsequences) to converge to $a$ in $L_2(\R_+,\rho)$.

\end{theorem}
We  remark that the $L_2(\R_+,\rho)$ used in our results is useful when $\rho$ has positive density on large intervals of $\R_+$.
Notice that the main result, under the validity of the coercivity condition, not only ensures the identification of $a$ on the support of $\rho$, but it also provides a prescribed rate of convergence.
For functions $a$ in $X_{M,K}$ and for finite element spaces $V_N$ of continuous piecewise linear functions constructed on regular meshes of size $N^{-1}$ a simple sequence $(a_N)_{N \in \mathbb N}$  realizing \eqref{rate1} with $\alpha=1$ and $C_1=M$ is the piecewise linear approximation to $a$ which interpolates $a$ on the mesh nodes. For the approximation estimate \eqref{rate2} there are plenty of results
concerning such rates and we refer to \cite{descsc13} and references therein. Roughly speaking, for $\mu_0^N$ the empirical measure obtained by sampling $N$ times independently from $\mu_0$, the bound \eqref{rate2} holds with high probability for a certain $N$ for $\beta$ of order $1/d$ (more precisely see \cite[Theorem 1]{descsc13}), which is a manifestation of the aforementioned curse of dimensionality. While it is in general relatively easy to increase $\alpha$ as the smoothness $a$ increases, and doing so independently of $d$, since $a$ is a function of one variable only, obtaining $\beta>1/d$ is in general not possible unless $\mu_0$ has very special properties, see \cite[Section 4.4 and Section 4.5]{fohavy11}.


\subsection{Numerical implementation of the variational approach }\label{sec:wp3}

The strength of the result from the variational approach followed in Section \ref{sec:wp2} is the total arbitrariness of the sequence $V_N$ except
for the assumed {\it uniform approximation property} and that the result holds - deterministically - with respect to the uniform convergence, which is quite strong.  However,
the condition that the spaces $V_N$ are to be picked as subsets of $X_{M,K}$ requires the prior knowledge of $M \geq \|a\|_{L_{\infty}(K)} + \|a'\|_{L_{\infty}(K)}$. Hence, the finite dimensional optimization \eqref{fdproxy} is not anymore a simple {\it unconstrained} least squares (as implicitly claimed in the paragraph before \eqref{fdproxy}),
but a problem constrained by a uniform bound on both the solution and its gradient. Nevertheless,
as we clarify in Section \ref{sec:num}, for $M>0$ fixed and choosing $V_N$ made of piecewise linear continuous functions, imposing the uniform $L_\infty$ bounds in the least square problem does not constitute a severe difficulty.
Also the tuning of the parameter $M>0$ turns out to be rather simple. In fact, for $N$ fixed the minimizers $\widehat a_N \equiv \widehat a_{N,M}$ have the property that the map
$$
 M \mapsto   \mathcal E^{[a],N}(\widehat a_{N,M})
$$
is monotonically decreasing as a function of the constraint parameter $M$ and it becomes constant for $M\geq M^*$, for $M^*>0$ empirically not depending on $N$. We claim that this special value $M^*$ is indeed the ``right" parameter for the $L_\infty$ bound. For such a choice, we show also numerically that, as expected, if we let $N$ grow, the minimizers $\widehat{a}_N$ approximates better and better the unknown potential $a$. 
\\

Despite the fact that both the tuning of $M>0$ and the constrained minimization over $X_{M,K}$  requiring $L_\infty$ bounds are not severe issues, it would be way more efficient to perform a unconstrained least squares over $X$. In our follow-up paper \cite{bofohamaXX} we extend the approach developed by Binev et al. in \cite{MR2249856,MR2327596} towards universal algorithms for learning regression functions from independent samples drawn according to an unknown probability distribution. This extension presents several challenges including the lack of independence of the samples collected in our framework and the nonlocality of the scalar products of the corresponding least squares. This has a price to pay, i.e., that  the spaces $V_N$ need to be carefully chosen and the result of convergence holds only with high probability.  
For the development of the latter results, we need to address several variational and measure theoretical properties of the model which are considered in details in this first paper as reported below.


\section{Preliminaries}\label{meanfield}
\subsection{Optimal transport and Wasserstein distances}

The space $\mathcal{P}(\R^n)$ is the set of probability measures on $\R^n$, while the space 
$\mathcal{P}_p(\R^n)$ is the subset of $\mathcal{P}(\R^n)$ whose elements $\mu$ have finite $p$-th moment, i.e.,
$\int_{\R^n} |x|^p d\mu(x) < +\infty$.
We denote by $\mathcal{P}_c(\R^n)$ the subset of $\mathcal{P}_p(\R^n)$ which consists of all probability measures with compact support. 
For any $\mu \in \mathcal{P}(\R^{n_1})$ and  a Borel function $f: \R^{n_1} \to \R^{n_2}$, we denote by $f_{\#}\mu \in \mathcal{P}(\R^{n_2})$ the {\it push-forward of $\mu$ through $f$}, defined by
\begin{align*}
f_{\#}\mu(B) := \mu(f^{-1}(B)) \quad \text{ for every Borel set } B \text{ of } \R^{n_2}.
\end{align*}
In particular, if one considers the projection operators $p_1$ and $p_2$ defined on the product space $\R^{n_1} \times \R^{n_2}$, for every $\pi \in \mathcal{P}(\R^{n_1} \times \R^{n_2})$ we call {\it first} (resp., {\it second}) {\it marginal} of $\pi$ the probability measure $p_{1\#}\pi$ (respectively, $p_{2\#}\pi$). Given $\mu \in \mathcal{P}(\R^{n_1})$ and $\nu \in \mathcal{P}(\R^{n_2})$, we denote with $\Gamma(\mu, \nu)$ the family of couplings between $\mu$ and $\nu$, i.e. the subset of all probability measures in $\mathcal{P}(\R^{n_1} \times \R^{n_2})$ with first marginal $\mu$ and second marginal $\nu$.

On the set $\mathcal{P}_p(\R^n)$ we shall consider the following distance, called the Wasserstein or Monge-Kantorovich-Rubinstein distance,
\begin{align}  \label{e_Wp}
\W^p_p(\mu,\nu)=\inf_{\pi \in \Gamma(\mu,\nu)} \int_{\R^{2n}} |x-y|^p d \pi(x,y)\,.
\end{align}
If $p = 1$, we have the following equivalent characterization of the $1$-Wasserstein distance:
\begin{align}\label{dualwass}
\W_1(\mu,\nu)=\sup \left \{ \int_{\R^n} \varphi(x) d (\mu-\nu)(x)  : \varphi \in \Lip(\R^n), \; \Lip_{\R^n}(\varphi) \leq 1 \right \},
\end{align}
where $\Lip_{\R^n}(\varphi)$ stands for the Lipschitz constant of $\varphi$ on $\R^n$. We denote by $\Gamma_o(\mu,\nu)$ the set of optimal couplings for which the minimum is attained, i.e.,
\begin{align*}
\pi \in \Gamma_o(\mu, \nu) \iff \pi \in \Gamma(\mu, \nu) \text{ and } \int_{\R^{2n}} | x - y |^p d \pi(x,y) = \W^p_p(\mu,\nu).
\end{align*}
It is well-known that $\Gamma_o(\mu, \nu)$ is non-empty for every $(\mu,\nu) \in \mathcal{P}_p(\R^n)\times\mathcal{P}_p(\R^n)$, hence the infimum in \eqref{e_Wp} is actually a minimum. For more details, see e.g. \cite{AGS,villani}.

For any $\mu \in \PP(\R^d)$ and $f: \R^d \to \R^d$, the notation $f * \mu$ stands for the convolution of $f$ and $\mu$:
\begin{align*}
(f * \mu)(x) = \int_{\R^d} f(x-y) d\mu(y)\,.
\end{align*}
This function is continuous and finite-valued whenever $f$ is continuous and \emph{sublinear}, i.e., there exists a constant $C > 0$ such that $| f(\xi) | \leq C (1 + |\xi|)$ for all $\xi \in \R^d$.

\subsection{The mean-field limit equation and existence of solutions}

As already stated in the introduction, our learning approach is based on the following underlying \textit{finite time horizon initial value problem}: given $T > 0$ and $\mu_0 \in \PC(\R^d)$, consider a probability measure-valued trajectory $\mu:[0,T]\rightarrow \PP(\R^d)$ satisfying 
\begin{align}\label{eq:contdyn}
\left\{\begin{aligned}
\frac{\partial \mu}{\partial t}(t) &= -\nabla \cdot ((\Fun{a}*\mu(t))\mu(t)) \quad \text{ for } t \in (0,T],\\
\mu(0) &=\mu_0.
\end{aligned}\right.
\end{align}
We consequently give our notion of solution for \eqref{eq:contdyn}.

\begin{definition}
We say that a map $\mu:[0,T]\rightarrow\PP(\R^d)$ is a solution of \eqref{eq:contdyn} with initial datum $\mu_0$ if the following hold:
\begin{enumerate}
\item $\mu$ has uniformly compact support, i.e., there exists $R > 0$ such that $\supp(\mu(t)) \subset B(0,R)$ for every $t \in [0,T]$;
\item $\mu$ is continuous with respect to the Wasserstein distance $\W_1$;
\item $\mu$ satisfies \eqref{eq:contdyn} in the weak sense, i.e., for every $\phi \in \mathcal{C}^{\infty}_c(\R^d;\R)$ it holds
\begin{align*}
\frac{d}{dt} \int_{\R^d} \phi(x) d\mu(t)(x) = \int_{\R^d} \nabla \phi(x) \cdot (\Fun{a}*\mu(t))(x) d\mu(t)(x).
\end{align*}
\end{enumerate}
\end{definition}

The equation \eqref{eq:contdyn} is closely related to the family of ODEs, indexed by $N \in \N$,
\begin{align}\label{eq:discrdyn}
\left\{\begin{aligned}
\dot{x}^N_i(t) &= \frac{1}{N}\sum^N_{j = 1}\Fun{a}(x^N_i(t) - x^N_j(t)) \quad \text{ for } t \in (0,T],\\
x_i^N(0) &= x^N_{0,i},
\end{aligned} \quad i = 1, \ldots, N, \right.
\end{align}
which may be rewritten as 
\begin{align}\label{eq:discr1}
\left\{\begin{aligned}
\dot{x}^N_i(t) &= (\Fun{a}*\mu^N(t))(x^N_i(t)) \\
x^N_i(0) &= x^N_{0,i},
\end{aligned} \quad i = 1, \ldots, N, \right.
\end{align}
for $t\in(0,T]$, by means of the \textit{empirical measure} $\mu^N:[0,T]\rightarrow\PC(\R^d)$ defined as
\begin{align}\label{eq:empmeas}
\mu^N(t) = \frac{1}{N}\sum^N_{j = 1} \delta_ {x^N_j(t)}.
\end{align}
As already explained in the introduction, we shall restrict our attention to interaction kernels belonging to the following \textit{set of admissible kernels}
\begin{align*}
	X=\bigl\{b:\R_+\rightarrow\R\,|\ b \in L_{\infty}(\R_+) \cap W^{1}_{\infty,\loc}(\R_+) \bigr \}.
\end{align*}
The well-posedness of \eqref{eq:discr1} is rather standard under the assumption $a \in X$. The well-posedness of system \eqref{eq:contdyn} and several crucial properties enjoyed by its solutions may also be proved as soon as $a \in X$.
We refer the reader to \cite{AGS} for results on existence and uniqueness of solutions for \eqref{eq:contdyn}, and to  \cite{13-Carrillo-Choi-Hauray-MFL} for generalizations in case of interaction kernels not necessarily belonging to the class $X$. Nevertheless, in the following we recall the main results, whose proofs are collected in the Appendices in order to keep this work self-contained and to allow explicit reference to constants.


\begin{proposition}\label{pr:exist}
Let $\mu_0 \in \PC(\R^d)$ be given. Let $(\mu^{N}_0)_{N \in \N} \subset \PC(\R^d)$ be a sequence of empirical measures of the form
\begin{align*}
\mu^{N}_0 = \frac{1}{N}\sum^N_{i = 1} \delta_{x^{N}_{0,i}}, \quad \text{ for some } x^{N}_{0,i} \in \supp(\mu_0)
\end{align*}
satisfying $\lim_{N \rightarrow \infty} \W_1(\mu_0,\mu^{N}_0) = 0$. For every $N \in \N$, denote with $\mu^N:[0,T] \rightarrow \PP(\R^{d})$ the curve given by \eqref{eq:empmeas} where $(x^N_1,\ldots,x^N_N)$ is the unique solution of system \eqref{eq:discrdyn}.

Then, there exists $R > 0$ depending only on $T,a$, and $\supp(\mu_0)$ such that the sequence $(\mu^N)_{N \in \N}$ converges, up to extraction of subsequences, in $\mathcal{P}_1(B(0,R))$ equipped with the Wasserstein metric $\W_1$ to a solution $\mu$ of \eqref{eq:contdyn} with initial datum $\mu_0$ satisfying
\begin{align*}
\supp(\mu^N(t)) \cup \supp(\mu(t)) \subseteq B(0,R), \quad \text{ for every } N \in \N \text{ and } t \in [0,T].
\end{align*}
\end{proposition}

 A proof of this standard result is reported in Appendix \ref{ap3} together with the necessary technical lemmas in Appendix \ref{ap1}.

\subsection{The transport map and  uniqueness of mean-field solutions}

Another way for building a solution of equation \eqref{eq:contdyn} is by means of the so-called \textit{transport map}, i.e., the function describing the evolution in time of the initial measure $\mu_0$. The transport map can be constructed by considering the following single-agent version of system \eqref{eq:discr1},
\begin{align}\label{eq:transpdyn}
\left\{\begin{aligned}
\dot{\xi}(t) &= (\Fun{a}*\mu(t))(\xi(t)) \quad \text{ for } t \in (0,T],\\
\xi(0) &= \xi_0,
\end{aligned}\right.
\end{align}
where $\xi$ is a mapping from $[0,T]$ to $\R^d$ and $a \in X$. Here $\mu:[0,T]\rightarrow\PP(\R^d)$ is a continuous map with respect to the Wasserstein distance $\W_1$ satisfying $\mu(0) = \mu_0$ and $\supp(\mu(t)) \subseteq B(0,R)$, for a given $R>0$.

According to  Proposition \ref{exmono}, if $\mu$ is any solution of \eqref{eq:contdyn}, we can consider the family of flow maps $\mathcal{T}^{\mu}_t:\R^d \rightarrow\R^d$, indexed by $t \in [0,T]$ and the mapping $\mu$, defined by
\begin{align*}
\mathcal{T}^{\mu}_t(\xi_0) = \xi(t),
\end{align*}
where $\xi:[0,T]\rightarrow\R^d$ is the unique solution of \eqref{eq:transpdyn} with initial datum $\xi_0$. The by now well-known result \cite[Theorem 3.10]{CanCarRos10} shows that the solution of \eqref{eq:contdyn} with initial value $\mu_0$ is also the unique fixed-point of the \textit{push-foward map}
\begin{align}\label{eq:fixedpoint}
\Gamma[\mu](t) := (\mathcal{T}^{\mu}_t)_{\#}\mu_0.
\end{align}

A relevant, basic property of the transport map is proved in the following

\begin{proposition}\label{p-transportlip}
$\ct^\mu_t$ is a locally bi-Lipschitz map, i.e. it is a locally Lipschitz map, with locally Lipschitz inverse.
\end{proposition}
\begin{proof}
	
	Let $R>0$ be sufficiently large such that $\supp(\mu_0)\subseteq B(0,R)$.
 The choice $r = R$ in Proposition \ref{le:uniquecara},  Lemma \ref{p-Floclip}, and  Lemma \ref{p-Fmuloclip}  imply the following stability estimate
	\begin{align}\label{eq:liptrans}
		\bigl|\ct^\mu_t(x_0)-\ct^\mu_t(x_1)\bigr|
			\leq e^{T\,\Lip_{B(0,R)}(\Fun{a})}|x_0-x_1|,\quad \text{ for } |x_i|\leq R\,,\quad i=0,1\,.
	\end{align}
	i.e., $\ct^\mu_t$ is locally Lipschitz.
	
	In view of the uniqueness of the solutions to the ODE \eqref{eq:transpdyn}, it is also clear that, for any $t_0\in [0,T]$, the inverse of $\ct^\mu_{t_0}$ is
	given by the transport map associated to the backward-in-time ODE
\begin{align*}
\left\{\begin{aligned}
\dot{\xi}(t) &= (\Fun{a}*\mu(t))(\xi(t)) \quad \text{ for } t \in [0,t_0),\\
\xi(t_0) &= \xi_0.
\end{aligned}\right.
\end{align*}
	However, this problem in turn can be cast into the form of an usual IVP simply by considering the reverse trajectory $\nu_t=\mu_{t_0-t}$. Then
	$y(t)=\xi(t_0-t)$ solves
	\begin{align*}
	\left\{\begin{aligned}
		\dot y(t)&=-\bigl(\Fun{a}\ast\nu(t)\bigr)(y(t))  \quad \text{ for } t \in (0,t_0], \\
		y(0) &= \xi(t_0).
	\end{aligned}\right.
	\end{align*}
	The corresponding stability estimate for this problem then yields that the inverse of $\ct^\mu_t$ exists and is locally Lipschitz (with the same local Lipschitz constant as $\ct^\mu_t$).
\end{proof}

It is also known, see, e.g., \cite{CanCarRos10}, that one has also uniqueness and continuous dependence on the initial data for \eqref{eq:contdyn} (we report a proof of these properties in Appendix \ref{ap3} for completeness):

\begin{theorem}\label{uniq}
Fix $T>0$  and let $\mu:[0,T]\rightarrow\mathcal{P}_1(\R^d)$ and $\nu:[0,T]\rightarrow\mathcal{P}_1(\R^d)$ be two equi-compactly supported solutions  of \eqref{eq:contdyn}, for $\mu(0)=\mu_0$ and $\nu(0)=\nu_0$ respectively. Let $R>0$ be such that for every $t \in[0, T]$
\begin{align}\label{supptot}
\supp(\mu(t))\cup\supp(\nu(t)) \subseteq B(0, R)\,.
\end{align}
Then, there exist a positive constant $\overline{C}$ depending only on $T$, $a$,  and $R$ such that
\begin{equation}\label{stab}
\W_1(\mu(t), \nu(t)) \le \overline{C} \, \W_1(\mu_0, \nu_0)
\end{equation}
for every $t \in [0, T]$. In particular, equi-compactly supported solutions of \eqref{eq:contdyn} are uniquely determined by the initial datum.
\end{theorem}


\section{The learning problem for the kernel function}\label{sec:learn}

As already explained in the introduction, our goal is to learn $a \in X$ from observation of the dynamics of $\mu^N$ corresponding to system \eqref{eq:discrdyn} with $a$ as interaction kernel, $\mu_0^N$ as initial datum, and $T$ as finite time horizon.

We pick $\widehat a$ among those functions in $X$ which would give rise to a dynamics close to $\mu^N$: roughly speaking we choose $\widehat a_N \in X$ as a minimizer of the following \textit{discrete error functional}
\begin{align}\label{eq-def-error}
	\begin{split}
	\Eahatn = \frac{1}{T}\int_0^T\frac{1}{N}\sum_{i=1}^N\biggl|\frac{1}{N}\sum_{j=1}^N
			\left(\widehat a(|\x_i(t)-\x_j(t)|)(\x_i(t) - \x_j(t))-\dot{x}^{[a]}_i(t)\right)\biggr|^2 dt.
	\end{split}
\end{align}

Let us remind that, by Proposition \ref{trajapprox}, this optimization guarantees also that any minimizer $\widehat a_N$ produces very good trajectory approximations $x^{[\widehat a_N]}(t)$ to the
``true" ones $\x(t)$ at least at finite time $t \in [0,T]$.

\begin{proof}[Proof of Proposition \ref{trajapprox}]
Let us denote $x=\x $ and $\widehat x =\xahat  $ and we estimate by Jensen or H\"older inequalities
\begin{align*}
\|x(t) - \widehat x(t) \|^2 & = \left \| \int_0^t ( \dot x(s) - \dot{\widehat x}(s)) ds \right \|^2 \leq  t \int_0^t \| \dot x(s) - \dot{\widehat x}(s) \|^2 ds \\
&=  t \int_0^t \frac{1}{N} \sum_{i=1}^N \left | (\Fun{a} * \mu^N(x_i)- \Fun{\widehat a} * \widehat \mu^N(\widehat x_i)) \right |^2 ds\\
&\leq 2 t \int_0^t \Bigg[  \frac{1}{N} \sum_{i=1}^N \left| (\Fun{a} - \Fun{\widehat a}) *  \mu^N( x_i)) \right |^2 \\
&\quad +\Bigg| \frac{1}{N} \sum_{j=1}^N \widehat a(|x_i-x_j|)( (\widehat x_j - x_j) + (x_i-\widehat x_i))  \\
&\quad+ \left(\widehat a(| \widehat x_i-\widehat x_j|) -  \widehat a(| x_i- x_j |)\right) (\widehat x_j - \widehat x_i) \Bigg|^2  \Bigg ] ds \\
&\leq 2 T^2 \Eahatn +  \int_0^t 8 T( \|\widehat a\|_{L_\infty(K)}^2 + (R \operatorname{Lip}_K(\widehat a) ) ^2 )\|x(s) - \widehat x(s) \|^2  ds,
\end{align*}
for $K=[0,2 R]$ and $R>0$ is as in Proposition \ref{pr:exist} for $a$ substituted by $\widehat a$.
An application of Gronwall's inequality yields the estimate
$$
\|x(t) - \widehat x(t) \|^2 \leq 2 T^2   e^{8 T^2( \|\widehat a\|_{L_\infty(K)}^2 + (R \operatorname{Lip}_K(\widehat a) ) ^2)} \Eahatn,
$$
which is the desired bound.
\end{proof}

\subsection{The measure $\prerho$}

In order to rigorously introduce the coercivity condition \eqref{eq-coercive}, we need to explore finer properties of the family of measures $(\varrho(t))_{t \in [0,T]}$, where we recall that 
$\varrho(t)(A)=(\mu(t)\otimes\mu(t))\bigl(d^{-1}(A)\bigr)$ for $A$ a Borel set of $\mathbb{R}_+$.

\begin{lemma}\label{rhosc}
	For every open set $A\subseteq\R_+$ the mapping $t \in [0,T] \mapsto\varrho(t)(A)$ is lower semi-continuous, whereas for
	any compact set $A$ it is upper semi-continuous.
\end{lemma}

\begin{proof}As a first step we show that for every given sequence $(t_n)_{n \in \N}$ converging to $t\in [0,T]$ we have the weak
	convergence $\varrho(t_n)\rightharpoonup\varrho(t)$ for $n \rightarrow \infty$. 
	We first note that $\mu(t_n)\otimes\mu(t_n)\rightharpoonup\mu(t)\otimes\mu(t)$, since $\mu(t_n)\rightharpoonup\mu(t)$ because of the continuity of $\mu(t)$ in the Wasserstein metric $\W_1$.	
%
	This implies the claimed weak convergence $\varrho(t_n)\rightharpoonup\varrho(t)$, since for any
	function $f\in \mathcal{C}(\R_+)$, it holds $f\circ d\in\mathcal{C}(\R^d\times\R^d)$, and hence
	\begin{align*}
		\int_{\R_+}f\,d\varrho(t_n)
			&=\int_{\R^{2d}}(f\circ d)(x,y)d(\mu(t_n)\otimes\mu(t_n))(x,y)\\
			&\stackrel{n\rightarrow\infty}{\longrightarrow}
				\int_{\R^{2d}}(f\circ d)(x,y)d(\mu(t)\otimes\mu(t))(x,y)
			=\int_{\R_+}f\,d\varrho(t).
	\end{align*}
	The claim now follows from general results for weakly* convergent sequences of Radon measures, see e.g. \cite[Proposition 1.62]{AFP00}.
\end{proof}

Lemma \ref{rhosc} justifies the following
\begin{definition}
The probability measure $\prerho$ on the Borel $\sigma$-algebra on $\R_+$ is defined for any Borel set $A \subseteq \R_+$ as follows
\begin{align}\label{eq-rho-4}
	\prerho(A):=\frac{1}{T}\int_0^T\varrho(t)(A)dt.
\end{align}
\end{definition}
Notice that Lemma \ref{rhosc} shows that \eqref{eq-rho-4} is well-defined only for sets $A$ that are open or compact in $\R_+$. This directly implies that $\prerho$ can be extended to any Borel set $A$, since both families of sets provide a basis for the Borel $\sigma$-algebra on $\R_+$. Moreover $\prerho$ is a regular measure on $\R_+$, since Lemma \ref{rhosc} also implies that for any Borel set $A$
\begin{align*}
	\prerho(A) =  \sup\{\prerho(F) : F \subseteq A, \;F \text{ compact}\} = \inf\{\prerho(G) : A \subseteq G, \;G \text{ open}\}\,.
\end{align*}

\vspace{0.3cm}

The measure $\prerho$ measures which - and how much - regions of $\R_+$ (the set of inter-point distances) are explored during the dynamics of the system. Highly explored regions are where our learning process ought to be successful, since these are the areas where we do have enough samples from the dynamics to reconstruct the function $a$.

We now show the absolute continuity of $ \prerho$ w.r.t. the Lebesgue measure on $\R_+$. First of all we observe the following:

\begin{lemma}\label{lemma-AC-1}
	Let $\mu_0$ be absolutely continuous w.r.t. the $d$-dimensional Lebesgue measure $\cl_d$. Then $\mu(t)$ is absolutely continuous w.r.t. $\cl_d$ for every $t\in [0,T]$.
\end{lemma}

\begin{proof}
Both $\mu_0$ and $\mu(t)$ are supported in $B(0,R)$, with $R$ as in \eqref{Rest}. The measure $\mu(t)$ is the pushforward of $\mu_0$ under the locally bi-Lipschitz map $\ct^\mu_t$, see Proposition \ref{p-transportlip}. Since $\ct^\mu_t$ has Lipschitz inverse on $B(0,R)$, this inverse maps $\cl_d$-null sets to $\cl_d$-null sets, so $\mu_0$-null sets are not only $\cl_d$-null sets by assumption, but are also $\mu(t)$-null sets.
\end{proof}

\begin{lemma}\label{le-abs}
	Let $\mu_0$ be absolutely continuous w.r.t. $\cl_d$. Then, for all $t\in [0,T]$, the measures $\varrho(t)$ and $\prerho$ are absolutely
	continuous w.r.t. $\cl_1\llcorner_{\R_+}$ (Lebesgue measure in $\R$ restricted to $\R_+$).
\end{lemma}

\begin{proof}
	Fix $t\in [0,T]$. By Lemma \ref{lemma-AC-1} we already know that $\mu(t)$ is absolutely continuous w.r.t.
	$\cl_d$, and so $\mu(t)\otimes\mu(t)$ is absolutely continuous w.r.t. $\cl_{2d}$. It hence
	remains to show that $\cl_{2d}$ is absolutely continuous w.r.t. $\cl_1\llcorner_{\R_+}$, where $d$ is the distance function,
	but this follows easily by observing that $d^{-1}(A)=0$ for every $\cl_1\llcorner_{\R_+}$-null set $A$, and an application of Fubini's theorem.	
	The absolute continuity of $\prerho$ now follows immediately from the one of $\varrho(t)$ for every $t$ and its
	definition as an integral average \eqref{eq-rho-4}.
\end{proof}

As an easy consequence of the fact that the dynamics of our system has support uniformly bounded in time, we get the following crucial properties of the measure $\prerho$.

\begin{lemma}\label{rhocompact} Let $\mu_0 \in \mathcal P_c(\mathbb R^d)$. Then
	the measure $\prerho$ is finite and has compact support.
\end{lemma}

\begin{proof}
We have
\begin{align*}
\begin{split}
\prerho(\R_+)&= \frac{1}{T}\int_0^T \varrho(t)(\R_+)dt 
= \frac{1}{T}\int_0^T \int_{\R^d \times \R^d} |x - y| d\mu(t)(x) d \mu(t)(y)dt
<+\infty,
\end{split}
\end{align*}
since the distance function is continuous and the support of $\mu$ is uniformly bounded in time. This shows that $\prerho$ is bounded.
Since the supports of the measures $\varrho(t)$ are the subsets of
$K=\{|x-y|:x,y\in B(0,R)\} = [0,2R]$, where $R$ is given by \eqref{Rest}, by construction we also have $\supp \prerho\subseteq K$.
\end{proof}

\begin{remark}
	While absolute continuity of $\mu_0$ implies the same for $\prerho$, the situation is different for purely atomic
	measures $\mu_0^N$: then $\mu^N(t)$ is also purely atomic for every $t$, and so it is
	$\varrho^N(t) = d_\# (\mu^N(t) \otimes \mu^N(t))$. However $\prerho$ is in general not purely atomic, due to the averaging in time in its definition \eqref{eq-rho-4}. 
	For
	example, one obtains
	\begin{align*}
		\frac{1}{T}\int_0^T\delta(t) dt=\frac{1}{T}\cl_1\llcorner_{[0,T]}\,,
	\end{align*}
	as becomes immediately clear when integrating a continuous function against those kind of measures.
\end{remark}

\subsection{On the coercivity assumption}\label{sec:coerc}

With the measure $\prerho$ at disposal, we define, as in \eqref{rho},  $$\rho(A) = \int_{A} s^2 d\prerho(s)$$ for all Borel sets $A \subset \R_+$.  An easy consequence of Lemma \ref{rhocompact} is that if $a\in X$, then
\begin{align}\label{eq:inftyimplyl2}
\|a\|^2_{L_2(\R_+,\rho)} = \int_{\R_+} \bigl|a(s)\bigr|^2 d\rho(s) \leq \rho(\mathbb R_+)\|a\|^2_{L_{\infty}(\supp(\rho))}\,,
\end{align}
and therefore $X\subseteq L_2(\R_+,\rho)$.  
As already mentioned in the introduction, for $N \to \infty$ a natural mean-field approximation to the learning functional is given by
\begin{align*}
	\Eahat=\frac{1}{T}\int_0^T \int_{\R^d} \biggl|\bigl((\Fun{\widehat a}-\Fun{a})\ast\mu(t)\bigr)(x)\biggr|^2d\mu(t)(x)dt,
\end{align*}
where $\mu(t)$ is a weak solution to \eqref{eq:contdyn}. 
By means of $\rho$, we recall the estimate from the Introduction
\begin{align}
\begin{split}\label{eq-rho-3}
	\Eahat&\leq\frac{1}{T}\int_0^T\int_{\R_+}\bigl|\widehat a(s)-a(s)\bigr|^2 s^2 d\varrho(t)(s) dt = \|\widehat a-a\|^2_{L_2(\R_+,\rho)}.
\end{split}
\end{align}
This inequality suggested in turn the coercivity condition \eqref{eq-coercive}: 
\begin{align*}
	\Eahat\geq c_T\|\widehat a-a\|^2_{L_2(\R_+,\rho)}.
\end{align*}
The main reason this condition is of interest to us is:

\begin{proposition}\label{uniquemin}
Assume $a \in X$ and that the coercivity condition \eqref{eq-coercive} holds. Then any minimizer of $\mathcal E^{[a]}$ in $X$ coincides $\rho$-a.e. with $a$.
\end{proposition}
\begin{proof}
Notice that $\mathcal E^{[a]}(a)=0$, and since $\Eahat\geq 0$ for all $\widehat a\in X$ this implies that $a$ is a minimizer of $\mathcal E^{[a]}$. Now suppose that $ \Eahat=0$ for some $\widehat a\in X$. By \eqref{eq-coercive} we obtain that $\widehat a=a$ in $L_2(\R_+,\rho)$, and therefore they coincide $\rho$-almost everywhere. 
\end{proof}

\subsubsection{Coercivity is ``generically'' satisfied}\label{randomod}
We make the case that while ``degeneracies'' would cause our coercivity condition to fail, i.e., $c_T=0$, in a ``generic'' case the coercivity inequality holds.
On the one hand, we show that if we could model the misfit $\mathcal K(r)=(a(r)- \widehat a(r))r$ to behave randomly, in a sufficiently independent manner, over a finite set of trajectory distances, then
the coercivity condition holds with high probability. While the needed independence assumptions will typically be  too strong to be applicable in practice, the arguments we provide are by far not the most general possible, and we view them as one possible notion of a ``generic'' case.
On the other hand, in the next section we also present a more rigorous {\it deterministic} argument to verify the coercivity condition for very particular choices of $a$.

With the notation of the misfit just introduced, the coercivity condition reads
\begin{align*}
\frac1T\int_0^T\int_{\R^d}&\left|\int_{\R^d} \mathcal K(|x-y|)\frac{x-y}{|x-y|}d\mu(t)(x)\right|^2d\mu(t)(y) \\
&\ge\frac{c_T}{T}\int_0^T\int_{\R^d}\int_{\R^d}\left| \mathcal K(|x-y|)\right|^2d\mu(t)(x)d\mu(t)(y)\,.
\end{align*}
If the inequality holds without the time average for a fixed $t_0$,
\begin{align*}
\int_{\R^d}\left|\int_{\R^d} \mathcal K(|x-y|)\frac{x-y}{|x-y|}d\mu(t_0)(x)\right|^2d\mu(t_0)(y)
\ge c_{t_0}'\int_{\R^d}\int_{\R^d}\left| \mathcal K(|x-y|)\right|^2d\mu(t_0)(x)d\mu(t_0)(y)\,,
\end{align*}
then by a continuity argument it can be extended to a nontrivial time interval. We will therefore freeze time and investigate the inequality at this fixed time. 
Additionally, for the moment we restrict our attention to the case where $\mu(t_0)$ is a discrete measure $\mu^N=\frac1N\sum_{i=1}^N\delta_{x_i}$ (we drop $t_0$ since it is now fixed), so that the inequality reads
\begin{align}
\frac1N\sum_{i=1}^N\left|\frac{1}{N} \sum_{j=1}^N  \mathcal K(|x_i-x_j|)\frac{x_i-x_j}{|x_i-x_j|}\right|^2\ge \frac{c'_{t_0}}{N^2}\sum_{i=1}^N\sum_{j=1}^N\left| \mathcal K(|x_i-x_j|)\right|^2\,.
\label{e:coercivitydiscrete}
\end{align}
We argue now that this (``instantaneous'') inequality holds with high probability as soon as the matrix $\mathbf{K}:=(\mathcal K(|x_i-x_j|))_{i,j=1,\dots,N}$ is modeled as a random matrix.
Although it is not completely plausible to argue statistical independence of the entries of such a matrix because it comes from evaluating a smooth function over distances of non-random points, this model is not completely unreasonable: after all $\mathbf K$ involves the difference of our estimator $\widehat a$ and the target influence function $a$. For least squares estimators this difference is random with the samples used to construct the estimator, often with nearly independent, perhaps even Gaussian, fluctuations. 
We assume that $\mathbf K$ has independent Gaussian rows, each with variance $\sigma^2I_N$. 
Since the bounds we wish to obtain, and our estimates below, are scale invariant, we may, and will, assume $\sigma=1$.
We now show that the coercivity assumption is satisfied, with a constant $c'_{t_0}=O(1/N)$. Let $\mathbf{X}_i\in\mathbb{R}^{N\times d}$ be the matrix whose $j$-th row is the (fixed) vector $\frac{x_i-x_j}{|x_i-x_j|}\in\mathbb{R}^d$, and let $\mathbf{K}(i,:)\in\mathbb{R}^N$ be the $i$-th row of $\mathbf{K}$. The coercivity inequality \eqref{e:coercivitydiscrete} may be re-written as:
\begin{align}
\frac 1N\sum_{i=1}^N\left|\frac1N\mathbf{K}(i,:)\mathbf{X}_i\right|^2\ge\frac{c'_t}{N^2}\|\mathbf{K}\|^2_{\mathbb{F}}\,.
\label{e:coermatrix}
\end{align}
Then we estimate
\begin{align*}
\mathbb{E}\left[|\mathbf{K}(i,:)\mathbf{X}_i|^2\right]
&=\sum_{l=1}^d\sum_{j,j'=1}^N\mathbb{E}\left[\mathcal{K}(|x_i-x_j|)\mathcal{K}(|x_i-x_{j'}|)|\right] \left(\frac{x_i-x_j}{|x_i-x_j|}\right)_l \left(\frac{x_i-x_{j'}}{|x_i-x_{j'}|}\right)_l\\
&=\sum_{l=1}^d\sum_{j=1}^N \mathbb{E}\left[\mathcal{K}(|x_i-x_j|)^2\right] \left(\frac{x_i-x_j}{|x_i-x_j|}\right)^2_l \\
&= \sum_{j=1}^N |\mathbf{X}_i(j,:)|^2 = \| \mathbf{X}_i\|^2_{\mathbb{F}}=N\,,
\end{align*}
where we used independence, 
and in the last step we used the fact that every row of $\mathbf{X}_i$ is a unit vector. By concentration one readily obtains that with high probability
\begin{align*}
\frac 1N\sum_{i=1}^N\left|\frac1N\mathbf{K}(i,:)\mathbf{X}_i\right|^2\ge\frac CN\,.
\end{align*}
One the other hand, since $\mathbb{E}[||\mathbf{K}||^2_{\mathbb{F}}]\le CN^2$ by standard random matrix theory results (e.g. \cite{Vershynin:NARMT}), and in fact not just in expectation but also with high probability, the right hand side of \eqref{e:coermatrix} is bounded by $c'_{t_0}C$ from above. Choosing $c'_{t_0}$ small enough (and at least as small as $O(1/N)$, as a function of $N$), we obtain \eqref{e:coermatrix} with high-probability.

The argument  may be generalized to other models of random matrices, for example with sub-Gaussian rows (for $\mathbf{K}$) and uniformly lower-bounded smallest singular values. One may also consider $\mathbf{X}_i$ random, sufficiently uncorrelated with $\mathbf{K}$, and obtain similar results. Also, the continuous case is not substantially different from the discrete case, as it may be derived by smoothing discrete approximations. We do not pursue these generalizations, as our purpose here is to show that the coercivity assumption is  ``generically'' satisfied.
A model where the behavior of the coercivity constant $c'_t$ would be quite different as $N$ grows, is the following: we assume that $\mathcal{K}(|x_i-x_j|)$ is distributed as $\frac{\eta_{ij}}{|x_i-x_j|^\alpha}$, where $\eta_{ij}$ are i.i.d. standard normal distributions, and furthermore we assume that as $N$ grows the quantity $\frac1N\sum_{j=1}^N |x_i-x_j|^{-\alpha}$ grows as $N^{\gamma-1}$, for some $\gamma\ge1$, and for every $i=1,\dots,N$ fixed. Repeating the calculation above we obtain that the coercivity condition holds with constant that scales as $O(N^{\gamma-1})$, in particular is $O(1)$ independently of $N$ for $\gamma=1$. The first assumption may be motivated that estimators of the influence function may have performance proportional to the gradient of the influence function itself, and such gradient may decay with distance; the second assumption is about the scaling of the ``bulk'' of the system as $N$ grows: for $\gamma=0$ such size is independent of $N$, for $\gamma>0$ it grows with $N$. Note that the case $\gamma=1$ is indeed very natural: the quantity $\frac1N\sum_{j=1}^N |x_i-x_j|^{-\alpha}$ is expected to approach the corresponding integral in the mean-field limit, which is independent of $N$. Under this natural scaling, the coercivity constant is independent of $N$, suggesting it holds in the limit as well.

\subsubsection{The deterministic case}
We construct now  deterministic examples of trajectories $t \to \mu(t)$ for which the  coercivity condition \eqref{eq-coercive} holds.
We start with the simple case of two particles, i.e., $N=2$, for which no specific assumptions on $a,\widehat a$ are required to verify \eqref{eq-coercive} other than their boundedness in $0$. Again it is convenient to write $\mathcal K(r) = (a(r) - \widehat a(r)) r$, so that the coercivity condition in this case can be reformulated as 
\begin{equation}\label{coercN2}
\frac{1}{T} \int_0^T \frac{1}{N} \sum_{i=1}^N \left | \frac{1}{N} \sum_{j=1}^N \mathcal K(|x_i-x_j|) \frac{x_i-x_j}{|x_i-x_j|} \right |^2 dt \geq \frac{c_T}{N^2T} \int_0^T  \sum_{i=1}^N \sum_{j=1}^N |\mathcal K(|x_i-x_j|)|^2  dt.
\end{equation}
Now, let us observe more closely the integrand on the left-hand-side, and for $\widehat i \neq i$, $i,  \widehat i \in \{1,2\}$ and $N=2$, and we obtain
\begin{eqnarray*}
\frac{1}{2} \sum_{i=1}^2 \left | \frac{1}{2} \sum_{j=1}^2 \mathcal K(|x_i-x_j|) \frac{x_i-x_j}{|x_i-x_j|} \right |^2 &=& \frac{1}{2} \sum_{i=1}^2 \left | \frac{1}{2} \sum_{j\neq i}^2 \mathcal K(|x_i-x_j|) \frac{x_i-x_j}{|x_i-x_j|} \right |^2 \\
&=&  \frac{1}{4} \sum_{i=1}^2 \left |  \mathcal K(|x_i-x_{\widehat i}|) \frac{x_i-x_{\widehat i}}{|x_i-x_{\widehat i}|} \right |^2\\
&=& \frac{1}{4} \sum_{i=1}^2 \left |  \mathcal K(|x_i-x_{\widehat i}|) \right |^2=\frac{1}{4}  \sum_{i=1}^2 \sum_{j=1}^2 |\mathcal K(|x_i-x_j|)|^2.
\end{eqnarray*}
Integrating over time the latter equality yields \eqref{coercN2} for $N=2$ with an actual equal sign and $c_T=1$. Notice that here we have not made any specific assumptions on the trajectories $t \mapsto x_i(t)$. 
Let us then consider the case of $N=3$ particles. Already in this simple case the angles between particles may be rather arbitrary and analyzing the many possible configurations becomes an involved exercise. (Notice that we circumvented this problem in the random model in Section \ref{randomod} thanks to the assumed independence of the entries of the rows of $\mathbf K$.)
To simplify the problem we assume that $d=2$  and that at a certain time $t$ the particles are disposed precisely at the vertexes of a equilateral triangle of edge length $r$. This makes the computation of the angles very simple. We also assume that $\mathcal K$ gets its maximal absolute value precisely at $r$, hence
$$
\frac{1}{9}   \sum_{i=1}^3 \sum_{j=1}^3 |\mathcal K(|x_i-x_j|)|^2  \leq \|\mathcal K\|_\infty^2 = \mathcal K(r)^2.
$$
Notice that, independently of the behavior of the particles at any other time $t \in [0,T]$, it holds also
\begin{equation}\label{maxbound}
\frac{1}{9 T} \int_0^T  \sum_{i=1}^3 \sum_{j=1}^3 |\mathcal K(|x_i-x_j|)|^2  dt \leq \|\mathcal K\|_\infty^2 = \mathcal K(r)^2.
\end{equation}
A direct computation in this case of particles disposed at the vertexes of a equilateral triangle shows that 
$$
 \frac{1}{3} \sum_{i=1}^3 \left | \frac{1}{3} \sum_{j=1}^3 \mathcal K(|x_i-x_j|) \frac{x_i-x_j}{|x_i-x_j|} \right |^2 =\frac{1}{3}  \mathcal K(r)^2,
$$
and therefore
$$ 
\frac{1}{3} \sum_{i=1}^3 \left | \frac{1}{3} \sum_{j=1}^3 \mathcal K(|x_i-x_j|) \frac{x_i-x_j}{|x_i-x_j|} \right |^2 \geq \frac{1}{18}   \sum_{i=1}^3 \sum_{j=1}^3 |\mathcal K(|x_i-x_j|)|^2.
$$
Unfortunately the assumption that $\mathcal K$ achieves its maximum in absolute value at $r$ does not  allow us yet to conclude by a simple integration over time the coercivity condition as we did for the case of two particles. In order to extend the validity of the inequality to arbitrary functions taking maxima at other points, we need to integrate over time by assuming now that the particles are vertexes of equilateral triangles with time dependent edge length, say from $r=0$ growing in time up to $r=2 R>0$. This will allow the trajectories to explore any possible distance within a given interval and to capture the maximal absolute value of  any kernel. More precisely, let us now assume that $\mathcal K$ is an arbitrary bounded continuous function,  achieving its maximal absolute value over $[0,2R]$, say at $r_0 \in (0,2R)$ and we can assume that this is obtained corresponding to the time $t_0$ when the particles form precisely the equilateral triangle of side length $r_0$. Now we need to make a stronger assumption on $\widehat a$, i.e., we require $\widehat a$ to belong to a class of equi-continuous functions, for instance functions which are Lipschitz continuous with uniform Lipschitz constant (such as the functions in $X_{M,K}$).
Under this equi-continuity assumption, there exist $\varepsilon>0$ and a constant $c_{T,\varepsilon}>0$ independent of $\mathcal K$ (but perhaps depending only on its modulus of continuity) such that
\begin{eqnarray*}\label{coercint}
&&\frac{1}{T} \int_0^T \frac{1}{3} \sum_{i=1}^3 \left | \frac{1}{3} \sum_{j=1}^3 \mathcal K(|x_i-x_j|) \frac{x_i-x_j}{|x_i-x_j|} \right |^2 dt \\\
&\geq& \frac{1}{T} \int_{t_0 - \varepsilon}^{t_0+\varepsilon} \frac{1}{3} \sum_{i=1}^3 \left | \frac{1}{3} \sum_{j=1}^3 \mathcal K(|x_i-x_j|) \frac{x_i-x_j}{|x_i-x_j|} \right |^2 dt\\
&\geq & \frac{c_{T,\varepsilon}}{3}  \mathcal K(r_0) \geq \frac{c_{T,\varepsilon}}{18T } \int_0^T  \sum_{i=1}^3 \sum_{j=1}^3 |\mathcal K(|x_i-x_j|)|^2  dt.
\end{eqnarray*}
In the latter inequality we used \eqref{maxbound}. Hence, also in this case, one can construct examples for which the coercivity assumption is verifiable. Actually this construction can be extended to any group of $N$ particles disposed on the vertexes of regular polygons. As an example of how one should proceed, let us consider the case of $N=4$ particles disposed instantanously at the vertexes of a square of side length $\sqrt{2} r>0$. In this case one directly verfies that
\begin{equation}\label{coerN4}
\frac{1}{4} \sum_{i=1}^4 \left | \frac{1}{4} \sum_{j=1}^4 \mathcal K(|x_i-x_j|) \frac{x_i-x_j}{|x_i-x_j|} \right |^2 = \frac{1}{16} ( \mathcal K(2 r) +  \sqrt 2 \mathcal K(\sqrt 2 r) )^2.
\end{equation}
Let us assume that the maximal absolute value of $\mathcal K$ is attained precisely at $\sqrt 2 r$. Then the minimum of the expression on the right-hand side of \eqref{coerN4} is attained for
the case where $\mathcal K(2 r)  = -  \mathcal K(\sqrt 2 r)$ yielding the following estimate from below
\begin{eqnarray*}
\frac{1}{4} \sum_{i=1}^4 \left | \frac{1}{4} \sum_{j=1}^4 \mathcal K(|x_i-x_j|) \frac{x_i-x_j}{|x_i-x_j|} \right |^2  &\geq& \frac{3 -2 \sqrt 2}{16} \mathcal K(\sqrt 2 r)^2.
\end{eqnarray*}
Hence, also in this case, we can apply the continuity argument above to eventually show the coercivity condition. Similar procedures can be followed for any $N \geq 5$. However, as $N \to \infty$ one can show numerically that the lower bound vanishes quite rapidly, making it impossible, perhaps not surprisingly, to conclude the coercivity condition for the uniform distribution over the circle.
\\

All the examples presented so far are based on discrete measures $\mu^N=\frac{1}{N} \sum_{i=1}^N\delta_{x_i}$ supported on particles lying on the vertexes of polytopes. However, one can consider an approximated convolution identity $g_\varepsilon$ for which $g_\varepsilon \to \delta_0$ for $\varepsilon \to 0$, where $\delta_0$ is a Dirac delta in $0$, and  the regularized probability measure
$$
\mu_\varepsilon(x) = g_\varepsilon * \mu^N (x)= \frac{1}{N} \sum_{i=1}^N g_\varepsilon (x-x_i).
$$
This diffuse measure approximates $\mu^N$ in the sense that $\mathcal W_1(\mu_\varepsilon,\mu^N) \to 0$ for $\varepsilon \to 0$, hence, in particular, integrals against Lipschitz functions can be well-approximated, i.e.,
$$
\left | \int_{\mathbb R^d} \varphi(x) d \mu^N(x) -\int_{\mathbb R^d}  \varphi(x) d \mu_\varepsilon(x) \right | \leq \Lip(\varphi) \mathcal W_1(\mu_\varepsilon,\mu^N) .
$$
Under the additional assumption that $\Lip_K(\widehat a) \sim \| \widehat a \|_{L_\infty(K)}$ (and this is true whenever $\widehat a$ is a piecewise polynomial function over a finite partition of $\R_+$, with the constant of the equivalence depending on the particular partition) one can extend the validity of the coercivity condition for $\mu^N$ \eqref{coercN2} 
to $\mu_\varepsilon$ as follows
\begin{align*}
&\frac{1}{T} \int_0^T \int_{\mathbb R^d}  \left | \int_{\mathbb R^d}  \mathcal K(|x-y|) \frac{y-x}{|y-x|} d \mu_\varepsilon(x) \right |^2 d\mu_\varepsilon(y) dt \\
& \geq  \frac{c_{T,\varepsilon }}{T} \int_0^T  \int_{\mathbb R^d}  \int_{\mathbb R^d}  |\mathcal K(|x-y|)|^2    d\mu_\varepsilon(x) d\mu_\varepsilon(y)dt,
\end{align*}
for a constant $c_{T,\varepsilon }>0$ for $\varepsilon>0$ small enough.
\\

In these latter sections we showed that the coercivity condition holds for ``generic" cases as well as for highly structured deterministic ones. In practice we can numerically verify that it holds in many situations, see Section \ref{numcoer}, and from now on we assume it without further concerns.

\subsection{Existence of minimizers of $\mathcal E^{[a],N}$}
The following proposition, which is a straightforward consequence of Ascoli-Arzel\'a Theorem, indicates the right ambient space where to state an existence result for the minimizers of $\mathcal E^{[a],N}$.

\begin{proposition}\label{XMdef}
Fix $M > 0$ and $K=[0,2R] \subset  \mathbb R_+$ for any $R>0$. Recall the set
\begin{align*}
X_{M,K} = \left\{b \in W^{1}_{\infty}(K) :
 \|b\|_{L_{\infty}(K)} + \|b'\|_{L_{\infty}(K)} \leq M
 \right\}.
\end{align*}
The space $X_{M,K}$ is relatively compact with respect to the uniform convergence on $K$.
\end{proposition}
\begin{proof}
Consider $(\widehat a_n)_{n \in \N} \subset X_{M,K}$. The Fundamental Theorem of Calculus (applicable for functions in $W^{1}_{\infty}$, see \cite[Theorem 2.8]{AFP00}) 
implies that the functions $\widehat a_n$ are all Lipschitz continuous with uniformly bounded Lipschitz constant, and are therefore equi-continuous. Since they are also pointwise uniformly equi-bounded,
by Ascoli-Arzel\'a Theorem there exists a subsequence converging uniformly on $K$ to some $\widehat a \in X_{M,K}$.
\end{proof}

\begin{proposition}\label{ENmin}
Assume $a \in X$. Fix $M > 0$ and $K=[0,2R] \subset  \mathbb R_+$ for $R>0$ as in Proposition \ref{pr:exist}. Let $V$ be a closed subset of $X_{M,K}$ w.r.t. the uniform convergence. Then, the optimization problem
\begin{align*}
	\min_{\widehat a \in V} \Eahatn
\end{align*}
admits a solution.
\end{proposition}
\begin{proof} For proving the statement we apply the direct method of calculus of variations.
Since $\inf \mathcal E^{[a],N} \geq 0$, we can consider a minimizing sequence $(\widehat a_n)_{n \in \N} \subset V$, i.e., such that $\lim_{n \rightarrow \infty}  \mathcal E^{[a],N} (\widehat a_n) = \inf_{V}  \mathcal E^{[a],N} $. By Proposition \ref{XMdef} there exists a subsequence of $(\widehat a_n)_{n \in \N}$ (labelled again $(\widehat a_n)_{n \in \N}$) converging uniformly on $K$ to a function $\widehat a \in V$ (since $V$ is closed). We now show that $\lim_{n \rightarrow \infty}  \mathcal E^{[a],N} (\widehat a_n) =  \mathcal E^{[a],N} (\widehat a)$, from which it follows  that $ \mathcal E^{[a],N} $ attains its minimum in $V$. 

As a first step, notice that the uniform convergence of $(\widehat a_n)_{n \in \N}$ to $\widehat a$ on $K$ and the compactness of $K$ imply that the functionals $\Fun{\widehat a_n}(x-y)$ converge uniformly to $\Fun{\widehat a}(x-y)$ on $B(0,R)\times B(0,R)$ (where $R$ is as in \eqref{Rest}). Moreover, we have the uniform bound
\begin{align}
\begin{split}\label{Faest}
\sup_{x,y\in B(0,R)}|\Fun{\widehat a_n}(x-y) - \Fun{a}(x-y)| &= \sup_{x,y\in B(0,R)}|\widehat a_n(|x-y|) -  a(|x-y|)| |x-y| \\
&\leq 2R \sup_{r\in K} |\widehat a_n(r) -  a(r)| \\
& \leq 2R(M + \|a\|_{L_{\infty}(K)}).
\end{split}
\end{align}
As the measures $\mu^N(t)$ are compactly supported in $B(0,R)$ uniformly in time, the boundedness \eqref{Faest} allows us to apply three times the dominated convergence theorem to yield
\begin{align*}
\lim_{n \rightarrow \infty}  \mathcal E^{[a],N} (\widehat a_n) &= \lim_{n \rightarrow \infty}\frac{1}{T}\int_0^T\int_{\R^d} \left| \int_{\R^d}
			\left(\Fun{\widehat a_n}(x-y)-\Fun{a}(x-y)\right)d\mu^N(t)(y)\right|^2d\mu^N(t)(x) dt\\
			&= \frac{1}{T}\int_0^T\lim_{n \rightarrow \infty}\int_{\R^d} \left| \int_{\R^d}
			\left(\Fun{\widehat a_n}(x-y)-\Fun{a}(x-y)\right)d\mu^N(t)(y)\right|^2 d\mu^N(t)(x) dt\\
			&= \frac{1}{T}\int_0^T\int_{\R^d} \left| \lim_{n \rightarrow \infty}\int_{\R^d}
			\left(\Fun{\widehat a_n}(x-y)-\Fun{a}(x-y)\right)d\mu^N(t)(y)\right|^2 d\mu^N(t)(x) dt\\
			&= \frac{1}{T}\int_0^T\int_{\R^d} \left| \int_{\R^d}
			\left(\Fun{\widehat a}(x-y)-\Fun{a}(x-y)\right)d\mu^N(t)(y)\right|^2 d\mu^N(t)(x) dt\\
&=  \Eahatn,
\end{align*}
which proves the statement.

\end{proof}
%


\section{$\Gamma$-convergence of $ \mathcal E^{[a],N} $ to $ \mathcal E^{[a]} $}


This section is devoted to a proof of Theorem \ref{thm}.

\subsection{Uniform convergence estimates}

We start with a technical result.

\begin{lemma}\label{lemma-semicontinuous-0}
	Under the assumptions of Theorem \ref{thm}, let $(b_N)_{N \in \N}\subset X_{M,K}$ be a sequence of continuous functions and $b \in X_{M,K}$, for $K=[0,2R]$  with $R>0$ as in \eqref{Rest}.
	Then we have the estimate
\begin{equation}\label{approxestimate}	
\bigl|\mathcal E^{[a],N}(b_N)-\mathcal E^{[a]}(b)\bigr|
		\leq  c_1 \mathcal{W}_1(\mu_0^N,\mu_0)+ c_2 \| b_{N}-b\|_{L_\infty(K)},
\end{equation}
where the constants are explicitly given by $c_1= 32 \overline{C}R(2R+1)M^2$ and $c_2 = 16R^2M$.
\end{lemma}

\begin{proof}
	By \eqref{stab}, $\mathcal{W}_1(\mu(t),\mu^N(t))\leq\overline{C}\mathcal{W}_1(\mu_0,\mu_0^N)$ uniformly in $t \in [0,T]$.
	For all $x,y,y' \in B(0,R)$, by the triangle inequality we have
	\begin{align*}
		|(\Fun{a} -\Fun{b})(x-y') - &(\Fun{a} -\Fun{b})(x- y)|  \\
			& \leq \left[2R (\Lip_K(a) + \Lip_K( b))    +\|a\|_{L_\infty(K)} + \| b\|_{L_\infty(K)} \right] |y-y'|,
	\end{align*}
	which implies the Lipschitz continuity of $(\Fun{a} -\Fun{b})(x- \cdot)$ in $B(0,R)$ for fixed $x\in B(0,R)$. Since $a,b \in X_{M,K}$, this implies
	\begin{equation}\label{lipfafb}
		\Lip_{B(0,R)}|(\Fun{a} -\Fun{b})(x-\cdot)|\leq 2 (2R+1)M,
	\end{equation}
	uniformly with respect to $x \in B(0,R)$. Consequently, we have
	\begin{align*}
		\biggl|\int_{\R^d}\bigl(\Fun{b}
			&-\Fun{a}\bigr)(x-y)d\mu^{N}(t)(y)-\int_{\R^d}\bigl(\Fun{b}-\Fun{a}\bigr)(x-y)d\mu(t)(y)\biggr|\\
			&\leq\Lip_{B(0,R)}|(\Fun{a} -\Fun{b})(x-\cdot)|\,\mathcal{W}_1(\mu^N(t),\mu(t))
				\leq 2 \overline{C}(2R+1)M\,\mathcal{W}_1(\mu_0^N,\mu_0),
	\end{align*}
	uniformly with respect to $t\in [0,T]$ and $x\in B(0,R)$. Furthermore, we also have
	\begin{eqnarray}
		\sup_{x,y \in B(0,R)}|\Fun{b_{N}}(x-y)-\Fun{b}(x-y)|&\leq& 2R \| b_{N}- b\|_{L_\infty(K)},\label{absbounda}\\
	\sup_{x,y \in B(0,R)}|\Fun{a}(x-y)-\Fun{b}(x-y)|	&\leq& 2R \| a- b\|_{L_\infty(K)}. \label{absbound}
	\end{eqnarray}
	Hence we further obtain
	\begin{align} \label{machecazzo}
		\Biggl|
			&\biggl|\int_{\R^d}\bigl(\Fun{b_{N}}-\Fun{a}\bigr)(x-y)d\mu^{N}(t)(y)\biggr| 
				-\biggl|\int_{\R^d}\bigl(\Fun{b}-\Fun{a}\bigr)(x-y)d\mu(t)(y)\biggr|\Biggr|\\ \nonumber
			&\leq\biggl|\int_{\R^d}\bigl(\Fun{b_{N}}-\Fun{a}\bigr)(x-y)d\mu^{N}(t)(y)
					-\int_{\R^d}\bigl(\Fun{b}-\Fun{a}\bigr)(x-y)d\mu(t)(y)\biggr|\\ \nonumber
			&\leq\Biggl|\int_{\R^d}
				\bigl(\Fun{b_{N}}-\Fun{b}\bigr)(x-y)d\mu^{N}(t)(y)\Biggr|\\ \nonumber
			&\qquad +\biggl|\int_{\R^d}\bigl(\Fun{b}-\Fun{a}\bigr)(x-y)d\mu^{N}(t)(y)
					-\int_{\R^d}\bigl(\Fun{b}-\Fun{a}\bigr)(x-y)d\mu(t)(y)\biggr|\\ \nonumber
			&\leq 2R \| b_{N}-b\|_{L_\infty(K)}\int_{\R^d}d\mu^{N}(t)(y)+ 2(2R+1)M\,\mathcal{W}_1(\mu^N(t),\mu(t))\\ \nonumber
			&\leq 2R \| b_{N}-b\|_{L_\infty(K)}+2 \overline{C}(2R+1)M\,\mathcal{W}_1(\mu_0^N,\mu_0). 
	\end{align}
Let
	\begin{align*}
		H_N(t,x)&=\Biggl|\int_{\R^d}\bigl(\Fun{b_{N}}-\Fun{a}\bigr)(x-y)d\mu^{N}(t)(y)\Biggr|^2\,,\quad
		&G_N(t)&= \int_{\R^d}H_N(t,x)d\mu^N(t)(x)\,,\\
		H(t,x)&= \Biggl|\int_{\R^d}\bigl(\Fun{b}-\Fun{a}\bigr)(x-y)d\mu(t)(y)\Biggr|^2\,,
		&G(t)&= \int_{\R^d}H(t,x)d\mu(t)(x).
	\end{align*}
	Then immediately it follows
	\begin{align}
		|G_N(t)-G(t)|&\leq \left|\int_{\R^d}H(t,x)d\mu^N(t)(x) - \int_{\R^d}H(t,x)d\mu(t)(x)\right| \nonumber\\
		&\quad + \int_{\R^d}\left|H_N(t,x) - H(t,x)\right|d\mu^N(t)(x). \label{aux3}
	\end{align}
From \eqref{absbound} and \eqref{lipfafb} we obtain
\begin{eqnarray*}
\Lip_{B(0,R)}H(t,\cdot)&\leq& 2 \left (\sup_{x,y \in B(0,R)}|\Fun{a}(x-y)-\Fun{b}(x-y)| \right ) \cdot\Lip_{B(0,R)}(\Fun{a} -\Fun{b})(\cdot -y) \\
&\leq& 4R\|a-b\|_{L_\infty(K)}\cdot 2(2R+1)M\,,
\end{eqnarray*}
and therefore
	\begin{align}\label{aux1}
		\left|\int_{\R^d}H(t,x)d\mu^N(t)(x) - \int_{\R^d}H(t,x)d\mu(t)(x)\right|
			&\leq\Lip_{B(0,R)}H(t,\cdot)\mathcal{W}_1(\mu^N(t),\mu(t))\nonumber\\
			&\leq 8R(2R+1)\overline{C}M\|a-b\|_{L_\infty(K)}\mathcal{W}_1(\mu_0^N,\mu_0)\nonumber\\
			&\leq 16R(2R+1)\overline{C}M^2 \mathcal{W}_1(\mu_0^N,\mu_0)
	\end{align}
	uniformly in $t \in [0,T]$. Similarly, \eqref{absbounda}, \eqref{lipfafb}, and \eqref{machecazzo} imply
	\begin{align}\label{aux2}
		\bigl|H_N(t,x)-H(t,x)\bigr|
			&\leq \Bigl(2R \| b_{N}-b\|_{L_\infty(K)}+2 \overline{C}(2R+1)M\,\mathcal{W}_1(\mu_0^N,\mu_0)\Bigr)\nonumber\\
			&\qquad\times
				2R\Bigl(\|b_N-a\|_{L_\infty(K)}+\|b-a\|_{L_\infty(K)}\Bigr)\nonumber\\
			&\leq 8RM\Bigl(2R \| b_{N}-b\|_{L_\infty(K)}+2\overline{C}(2R+1)M\,\mathcal{W}_1(\mu_0^N,\mu_0)\Bigr)
	\end{align}
	uniformly in $t \in [0,T]$ and $x \in B(0,R)$. A combination of \eqref{aux3} with \eqref{aux1} and
	\eqref{aux2} yields
	\begin{align*}
		|G_N(t)-G(t)|
			\leq 32\overline{C}R(2R+1)M^2 \mathcal{W}_1(\mu_0^N,\mu_0)+16R^2M\| b_{N}-b\|_{L_\infty(K)}
	\end{align*}
	uniformly in $t \in [0,T]$. Thus we finally arrive at
	\begin{align*}
		\bigl|\mathcal E^{[a],N}(b_N)-\mathcal E^{[a]}(b)\bigr|
			&=\biggl|\frac{1}{T}\int^T_0\bigl(G_N(t)-G(t)\bigr)dt\biggr|\\
			&\leq 32\overline{C}R(2R+1)M^2 \mathcal{W}_1(\mu_0^N,\mu_0)+16R^2M\| b_{N}-b\|_{L_\infty(K)}.
	\end{align*}
	This proves the claim.

\end{proof}

As a corollary, we now immediately obtain the following convergence result.

\begin{lemma}\label{lemma-semicontinuous-1}
	Under the assumptions of Theorem \ref{thm}, let $(b_N)_{N \in \N}\subset X_{M,K}$ be a sequence of continuous functions
	uniformly converging to a function $b \in X_{M,K}$ on $K=[0,2R]$  with $R>0$ as in \eqref{Rest}.
	Then it holds
	\begin{align*}
		\lim_{N\rightarrow\infty} \mathcal E^{[a],N}(b_{N})= \mathcal E^{[a]}(b).
	\end{align*}
\end{lemma}

\begin{proof}
	This follows immediately from the estimate \eqref{approxestimate}, upon noticing
	$\W_1(\mu_0,\mu^N_0) \rightarrow 0$ for $N \rightarrow \infty$ as a consequence of the Glivenko-Cantelli theorem,
see for instance \cite[Lemma 3.3]{fornahuetter}.
\end{proof}

\subsection{Proof of the main result}

We are now ready to present the proof of our main result Theorem \ref{thm}.

\begin{proof}[\normalfont\bf Proof of Theorem \ref{thm}]

	The sequence of minimizers $(\widehat a_N)_{N \in \N}$ is by definition a subset of $X_{M,K}$, hence by Proposition \ref{XMdef} it admits a subsequence $(\widehat a_{N_k})_{k \in \N}$ uniformly converging to a function $\widehat a \in X_{M,K}$.
	
	To show the optimality of $\widehat a$ in $X_{M,K}$, let $b\in X_{M,K}$ be given. By Definition \ref{VNdef}, we can find a sequence $(b_N)_{N \in \N}$ converging uniformly to $b$ on $K$ such that $b_N\in V_N$ for every $N\in \N$. Lemma \ref{lemma-semicontinuous-1} implies
	\begin{align*}
		\lim_{N\rightarrow\infty} \mathcal E^{[a],N}(b_{N})= \mathcal E^{[a]}(b),
	\end{align*}	
	and, by the optimality of $\widehat a_{N_k}$ in $V_N$, it follows that
	\begin{align*}
		\mathcal E^{[a]} (b)=\lim_{N\rightarrow\infty}\mathcal E^{[a],N}(b_N)
			= \lim_{k \rightarrow\infty}\mathcal E^{[a],N_k}(b_{N_k})
			\geq\lim_{k \rightarrow\infty}\mathcal E^{[a],N_k}(\widehat a_{N_k})
			= \mathcal E^{[a]} (\widehat a)\,.
	\end{align*}
	We can therefore conclude that for every $b \in X_{M,K}$
	\begin{align}\label{fond}
		 \mathcal E^{[a]} (b)\geq \mathcal E^{[a]} (\widehat a)\,.
	\end{align}
 In particular, \eqref{fond} applies to $b=a\in X_{M,K}$ (by the particular choice of $M$), which finally implies
	\begin{align*}
		0=\mathcal E^{[a]} (a)\geq \mathcal E^{[a]} (\widehat a)\geq 0 \mbox { or } \mathcal E^{[a]} (\widehat a)=0,
	\end{align*}
	showing that $\widehat a$ is also a minimizer of $\mathcal E^{[a]}$. When the coercivity condition \eqref{eq-coercive} holds,  it follows that $\widehat a=a$ in $L_2(\R_+,\rho)$.
Assume now that \eqref{rate1} and \eqref{rate2} hold together with  \eqref{eq-coercive}. Then, by these latter conditions, two applications of \eqref{approxestimate}, the minimality of $\widehat a_N$, and the optimality of
$a$ in the sense that $\mathcal E^{[a]}(a)=0$, we obtain the following chain of estimates
\begin{eqnarray*}
\| \widehat a_N - a \|_{L_2(\R_+,\rho)}^2 &\leq& \frac{1}{c_T}  \mathcal E^{[a]}(\widehat a_N) \\
&\leq& \frac{1}{c_T} \left ( \mathcal E^{[a],N}(\widehat a_N) + ( \mathcal E^{[a]}(\widehat a_N)- \mathcal E^{[a],N}(\widehat a_N))  \right )\\
&\leq & \frac{1}{c_T} \left (\mathcal E^{[a],N}(\widehat a_N) +  c_1 \mathcal W_1(\mu_0^N,\mu_0) \right ) \\
&\leq &\frac{1}{c_T} \left (\mathcal E^{[a],N}(a_N) +  c_1 \mathcal W_1(\mu_0^N,\mu_0) \right ) \\
&\leq &\frac{1}{c_T} \left (2 c_1 \mathcal W_1(\mu_0^N,\mu_0) + c_2 \| a - a_N \|_{L_\infty(K)} \right ) \\
&\leq & C_3  N^{-\min \{\alpha,\beta\}}.
\end{eqnarray*} 
This concludes the proof.
\end{proof}


\section{Numerical experiments}\label{sec:num}

In this section we report several numerical experiments to document the validity and applicability of Theorem \ref{thm}. We will first show how the reconstruction of the unknown kernel $a$ gets better as the number of agents $N$ increases, in accordance with the $\Gamma$-convergence result reported in the last section. This feature holds true also for at least some interaction kernels not lying in the function space $X$, as shown in Figure \ref{variableN2}. We will then investigate empirically the validity of the coercivity condition \eqref{eq-coercive} comparing the functional $\mathcal E^{[a],N}(\widehat a_N)$ with $\|a - \widehat a_N\|_{L_2(\R_+,\rho^N)}^2$ where $\rho^N$ is constructed as $\rho$ but referring to the empirical measures $\mu^N$. Then we address  the behavior of $\mathcal E^{[a],N}(\widehat a_{N,M})$ for $N$ fixed, while we let the constraint constant  $M$ vary (here $\widehat a_{N,M}\equiv\widehat a_{N}$). Finally, we show how we can get a very satisfactory reconstruction of the unknown interaction kernel by keeping $N$ fixed and averaging the minimizers of the functional $\mathcal E^{[a],N}$ obtained from several samples of the initial data distribution $\mu_0$.

\subsection{Numerical framework}\label{numfram}

All experiments  rely on a common numerical set-up, which we clarify in this section. All the initial data $\mu^N_0$ are drawn from a common probability distribution $\mu_0$ which is the uniform distribution on the $d$-dimensional cube $[-L,L]^d$. For every $\mu^N_0$, we simulate the evolution of the system starting from $\mu^N_0$ until time $T$, and we shall denote with $R$ the maximal distance between particles reached during the time frame $[0,T]$. Notice that we have at our disposal only a finite sequence of snapshots of the dynamics: if we denote with $0 = t_0 < t_1 < \ldots < t_m = T$ the time instants at which these snapshots are taken, we can consider the \textit{discrete-time error functional}
\begin{align*}
\mathcal{E}^{[a],N}_\Delta(\widehat{a}) & = \frac{1}{m} \sum^m_{k = 1} \frac{1}{N} \sum^N_{j = 1} \left| \frac{1}{N} \sum^N_{i = 1} \widehat{a}(|x_j(t_k) - x_i(t_k)|)(x_j(t_k) - x_i(t_k)) - \dot{x}_i(t_k)\right|^2,
\end{align*}
which is the time-discrete counterpart of the continuous-time error functional $\mathcal{E}^{[a],N}$. As already mentioned in the introduction, derivatives $\dot{x}_i(t_k)$ appearing in $\mathcal{E}^{[a],N}_\Delta$ are actually approximated as well by finite differences: in our experiments we will use the simplest approximation
\begin{align*}
\dot{x}_i(t_k) = \frac{x_i(t_k) - x_i(t_{k-1})}{t_k - t_{k-1}}, \text{ for every } k \geq 1.
\end{align*}

Regarding the reconstruction procedure, we fix the constraint level $M>0$ and consider the sequence of invading subspaces $V_N$ of $X_{M,K}$ ($K=[0,2 R]$ here) generated by a B-spline basis with $D(N) $ elements supported on $[0,2 R]$: for every element $\widehat{a} \in V_N$ it holds
\begin{align*}
	\widehat{a}(r) = \sum^{D(N)}_{\lambda = 1} a_{\lambda} \varphi_{\lambda}(r), \qquad r \in [0,2R].
\end{align*}
In order for $V_N$ to increase in $N$ and invade $X_{M,K}$, we let $D(N)$ be a strictly increasing function of $N$. For the sake of simplicity, we shall employ a \textit{linear uniform} B-spline basis supported on the interval $[0,2 R]$ with $0$-smoothness conditions at the boundary, see \cite{deboor}.

Whenever $\widehat{a} \in V_N$, we can rewrite the functional $\mathcal{E}^{[a],N}_\Delta$ as
\begin{align*}
\mathcal{E}^{[a],N}_\Delta(\widehat{a}) & = \frac{1}{m} \sum^m_{k = 1} \frac{1}{N} \sum^N_{j = 1} \left| \frac{1}{N} \sum^N_{i = 1} \sum^{D(N)}_{\lambda = 1} a_{\lambda} \varphi_{\lambda}(|x_j(t_k) - x_i(t_k)|)(x_j(t_k) - x_i(t_k)) - \dot{x}_i(t_k)\right|^2 \\
& = \frac{1}{m} \sum^m_{k = 1} \frac{1}{N} \sum^N_{j = 1} \left| \sum^{D(N)}_{\lambda = 1} a_{\lambda} \frac{1}{N} \sum^N_{i = 1} \varphi_{\lambda}(|x_j(t_k) - x_i(t_k)|)(x_j(t_k) - x_i(t_k)) - \dot{x}_i(t_k)\right|^2 \\
& = \frac{1}{mN} \left\| \mathbf C  \vec{a} - \vec v \right\|^2_{2},
\end{align*}
where $\vec{a} = (a_1, \ldots, a_{D(N)})$, $\vec v = (\dot{x}_1(t_1), \ldots, \dot{x}_N(t_1), \ldots,\dot{x}_1(t_m), \ldots, \dot{x}_N(t_m))$ and the tensor $\mathbf C \in \R^{d\times Nm \times D(N)}$ satisfies for every $j = 1, \ldots,N$, $k = 1, \ldots,m$, $\lambda = 1, \ldots,D(N)$
\begin{align*}
\mathbf C(jk,\lambda) = \frac{1}{N} \sum^N_{i = 1} \varphi_{\lambda}(|x_j(t_k) - x_i(t_k)|)(x_j(t_k) - x_i(t_k)) \in \R^d.
\end{align*}

We shall numerically implement the constrained minimization with the software CVX \cite{cvx,gb08}, which allows constraints and objectives to be specified using standard MATLAB expression syntax. In order to use it, we need to rewrite the constraint of our minimization problem, which reads
\begin{align*}
	\|a\|_{L_{\infty}([0,R])} + \|a'\|_{L_{\infty}([0,R])} \leq M,
\end{align*}
using only the minimization variable of the problem, which is the vector of coefficients of the B-spline basis $\vec{a}$. Notice that the property of being a linear B-spline basis implies that, for every $\lambda = 1, \ldots, D(N)-1$, the property $\supp(\varphi_{\lambda}) \cap \supp(\varphi_{\lambda + j}) \not= \emptyset$ holds if and only if $j = 1$. Hence, for every $a \in V_N$ we have
\begin{align*}
\|a\|_{L_{\infty}([0,2 R])} &= \max_{r \in [0,R]} \left|\sum^{D(N)}_{\lambda = 1} a_{\lambda} \varphi_{\lambda}(r)\right|
 \leq \max_{\lambda = 1, \ldots, D(N)-1} \left(|a_{\lambda}| + |a_{\lambda+1}|\right) 
 \leq 2 \|\vec{a}\|_{\infty},
\end{align*}
\begin{align*}
\|a'\|_{L_{\infty}([0,2 R])} &= \max_{r \in [0,2 R]} \left|\sum^{D(N)}_{\lambda = 1} a_{\lambda} \varphi'_{\lambda}(r)\right| 
 \leq \max_{\lambda = 1, \ldots, D(N)-1} |a_{\lambda+1} - a_{\lambda}|
 = \|\mathbf D\vec{a}\|_{\infty},
\end{align*}
where, in the last line, $\mathbf  D$ is the standard finite difference matrix
\begin{align*}
\mathbf  D = \begin{bmatrix}
    1       & -1 & 0 & \dots & 0 & 0 \\
    0     & 1 & -1 & \dots & 0 & 0 \\
   \vdots & \vdots & \ddots & \ddots & \vdots & \vdots \\ \\
    0       & 0 & 0 & \dots & 1 & -1 \\
    0       & 0 & 0 & \dots & 0 & 0
\end{bmatrix}.
\end{align*}
We therefore replace the constrained minimization problem
\begin{align*}
\min_{\widehat{a} \in V_N} \mathcal{E}^{[a],N}(\widehat{a}) \quad \text{ subject to } \quad \|\widehat{a}\|_{L_{\infty}([0,R])} + \|\widehat{a}'\|_{L_{\infty}([0,R])} \leq M,
\end{align*}
by
\begin{align}\label{problem2}
\min_{\vec{a} \in \R^{D(N)}} \frac{1}{mN} \left\|\mathbf C \vec{a} - \vec v \right\|^2_{2} \quad \text{ subject to } \quad 2\|\vec{a}\|_{\infty} + \|\mathbf D\vec{a}\|_{\infty} \leq M\,,
\end{align}
which has weaker constraints, but is amenable to numerical solution.
The byproduct of the time discretization and the reformulation of the constraint is that minimizers of problem \eqref{problem2} may not be precisely the minimizers of the original one. This is the price to pay for this simple numerical implementation of the $L_\infty$-constraints. Despite such a crude discrete model,  we still observe all the approximation properties proved in the previous sections and the implementation results both simple and effective.

\subsection{Varying $N$}

In Figure \ref{variableN} we show the reconstruction of a truncated Lennard-Jones type interaction kernel obtained with different values of $N$. Table \ref{tab:fig1} reports the values of the different parameters.

\begin{table}[h!]
\begin{center}
\begin{tabular}{ |c|c|c|c|c|c| }
\hline
  $d$ & $L$ & $T$ & $M$ & $N$ & $D(N)$ \\
\hline
\hline
  $2$ & $3$ & $0.5$ & $100$ & $[10,20,40,80]$ & $2N$ \\
\hline
\end{tabular}
\end{center}
\vspace{-0.5cm}
\caption{Parameter values for Figure \ref{variableN} and Figure \ref{variableN2}.} \label{tab:fig1} 
\end{table}

It is clearly visible how the the piecewise linear approximant (displayed in blue) gets closer and closer to the potential to be recovered (in red), as predicted by the theoretical results of the previous sections. What is however surprising is that the same behavior is witnessed in Figure \ref{variableN2}, where the algorithm is applied to an interaction kernel $a$ not belonging to the function space $X$ (due to its singularity at the origin) with the same specifications reported in Table \ref{tab:fig1}. In particular, the algorithm performs an excellent approximation despite the highly oscillatory nature of the function $a$ and produce a natural numerical homogeneization when the discretization is not fine enough.

\begin{figure}[h!]
\begin{center}
\hspace{-0.7cm}\includegraphics[width=0.55\textwidth]{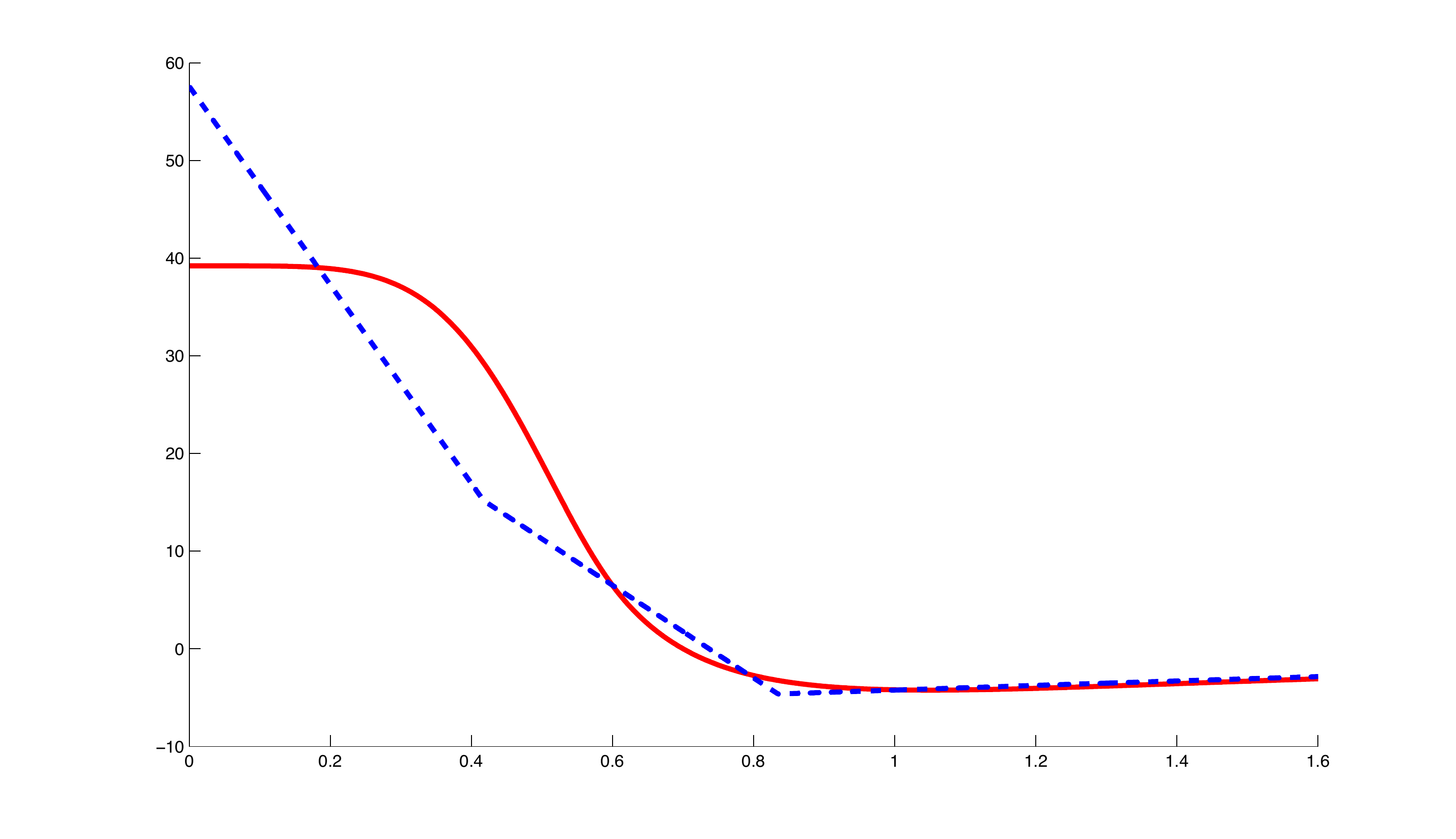}\hspace{-0.9cm}
\includegraphics[width=0.55\textwidth]{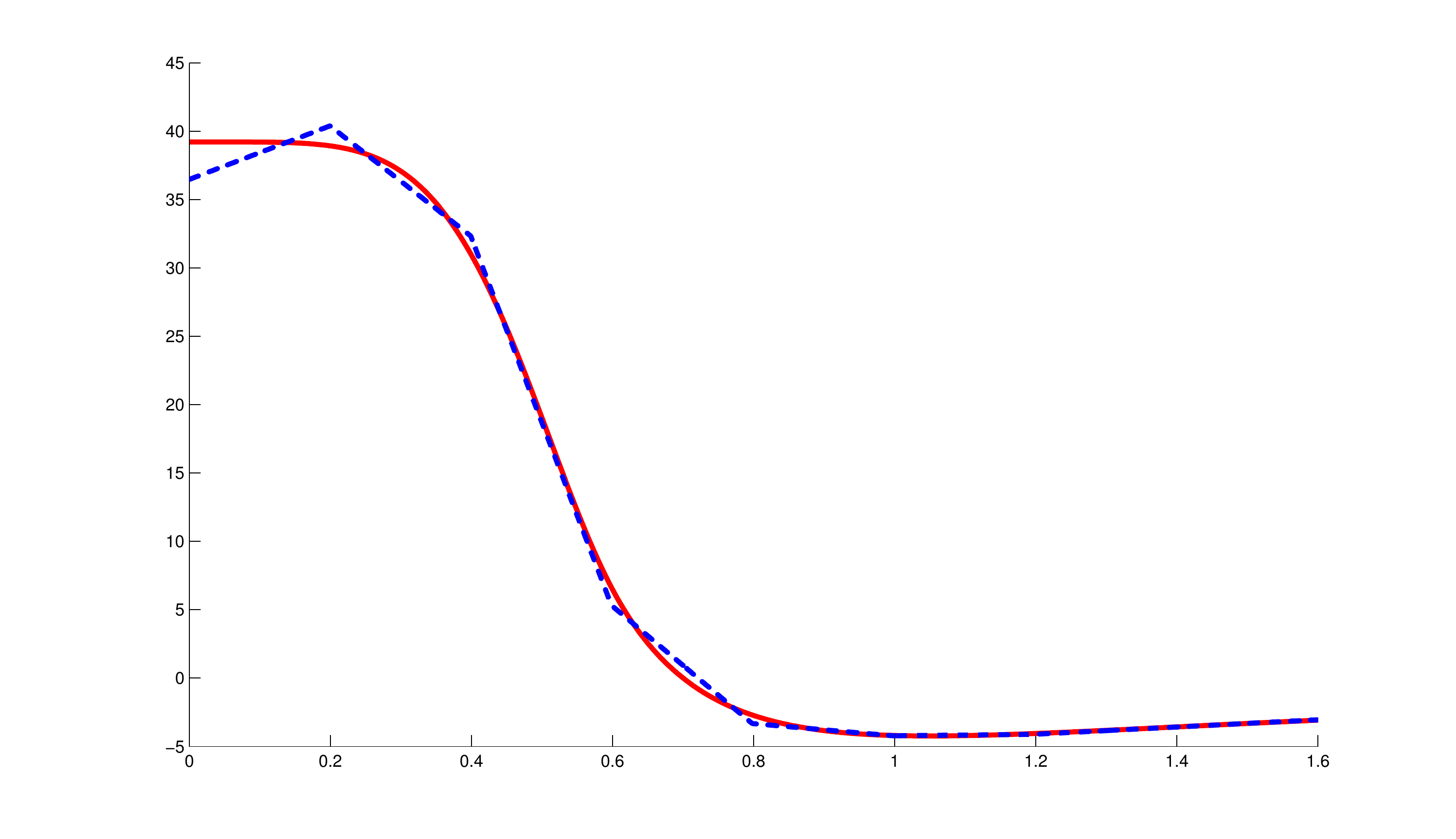}\\
\hspace{-0.7cm}\includegraphics[width=0.55\textwidth]{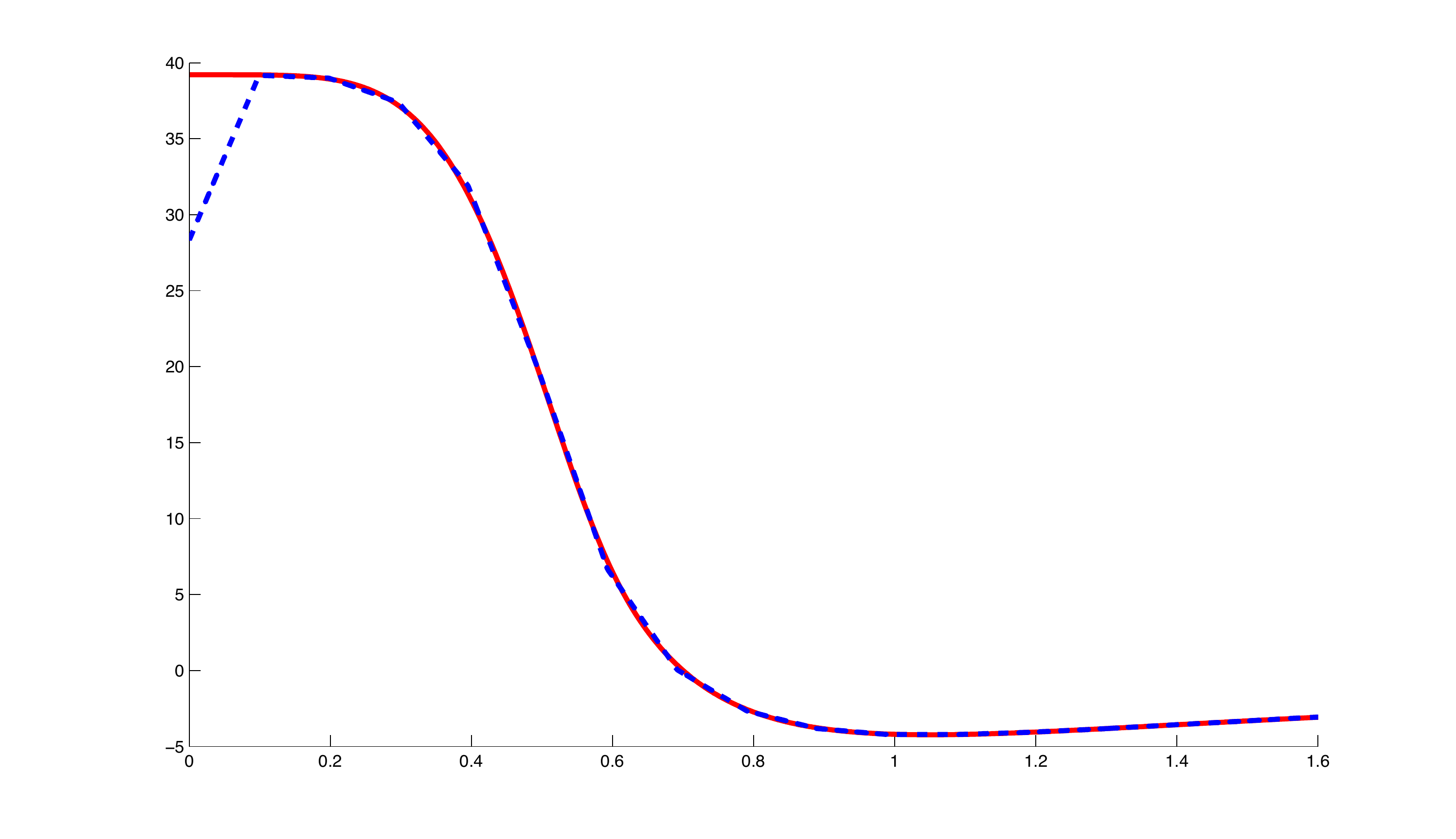}\hspace{-0.9cm}
\includegraphics[width=0.55\textwidth]{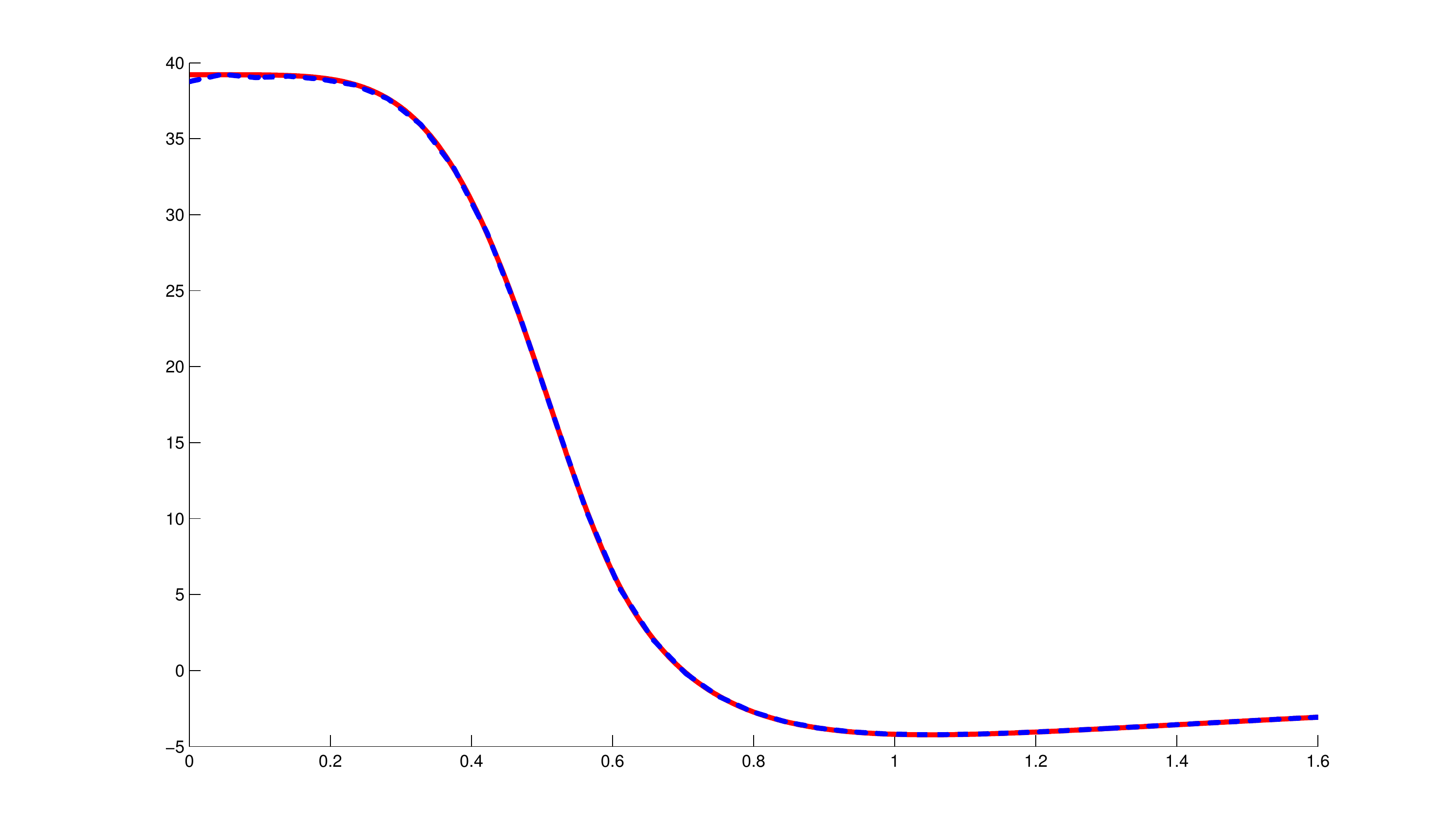}
\end{center}
\caption{Iterative reconstruction of a potential with different values of $N$. In red: the unknown kernel. In blue: its reconstruction by minimization of $\mathcal{E}^{[a],N}$. From left-top to right-bottom: reconstruction with $N = 10, 20, 40, 80$ agents. We notice that the uniform convergence at $0$ is slower in view of the quadratic polynomial weight $s^2$ as in \eqref{rho} and because less information is actually gathered around $0$.}\label{variableN}
\end{figure}

\begin{figure}[h!]
\begin{center}
\hspace{-0.7cm}\includegraphics[width=0.55\textwidth]{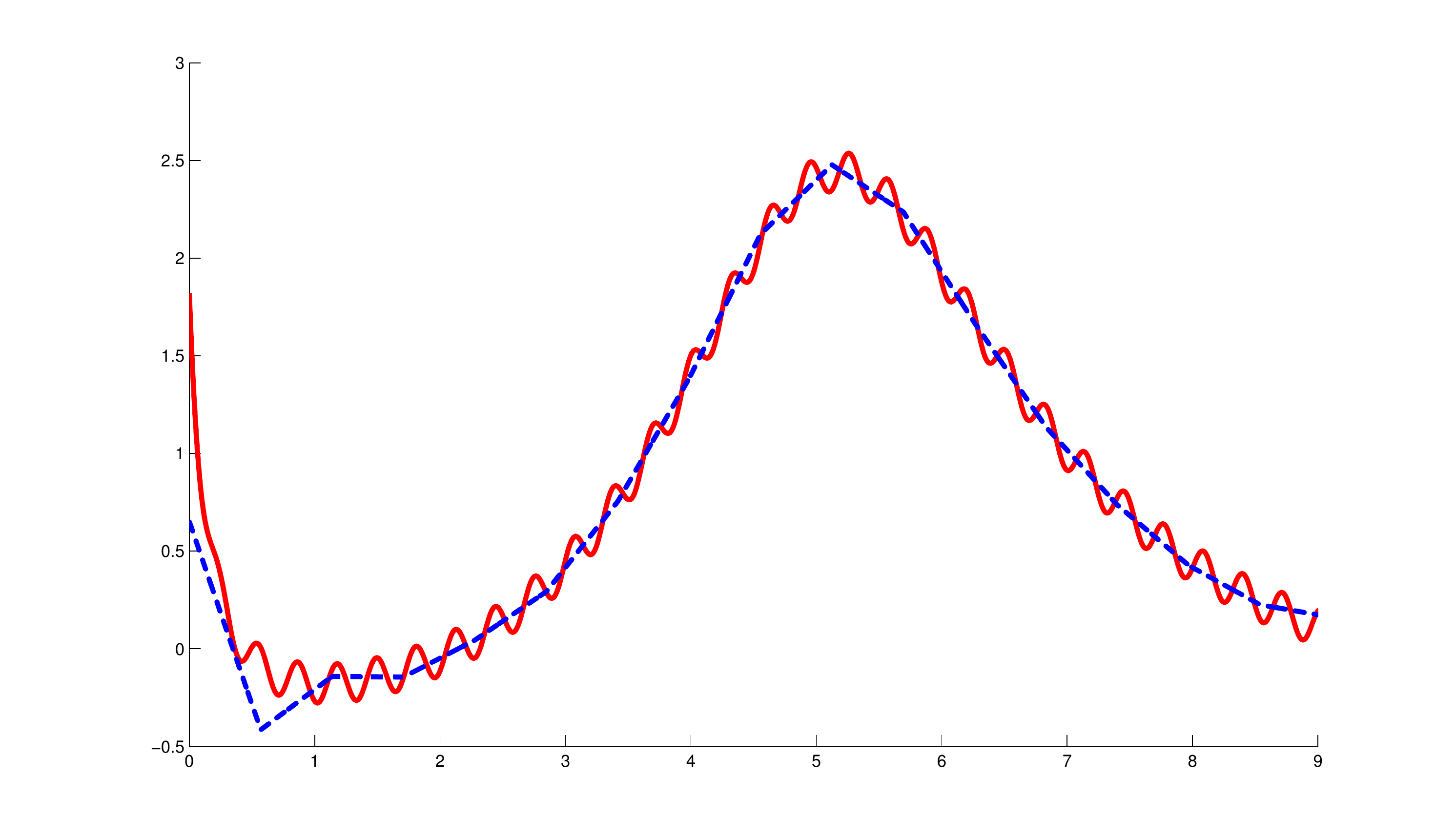}\hspace{-0.9cm}
\includegraphics[width=0.55\textwidth]{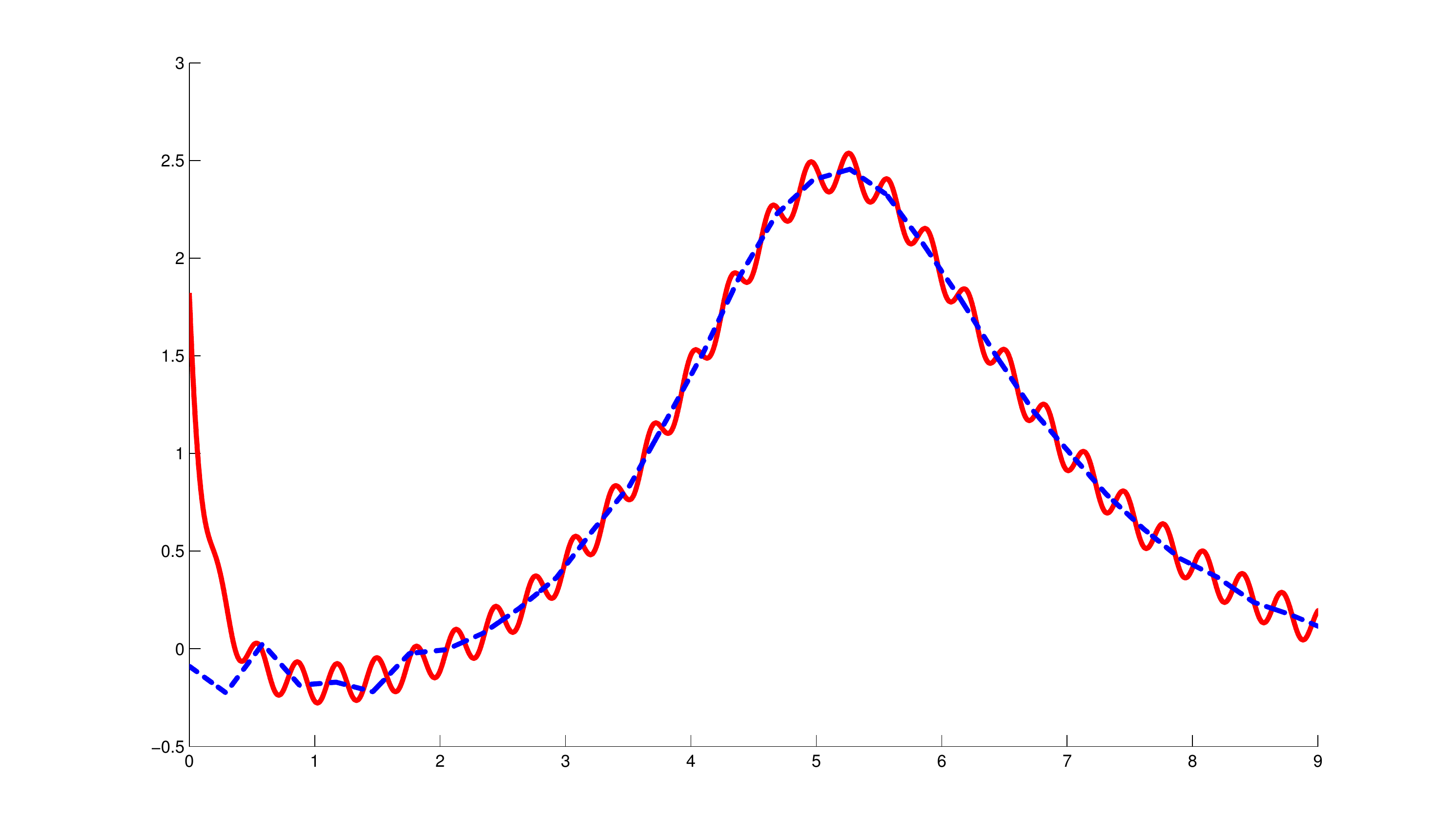}\\
\hspace{-0.7cm}\includegraphics[width=0.55\textwidth]{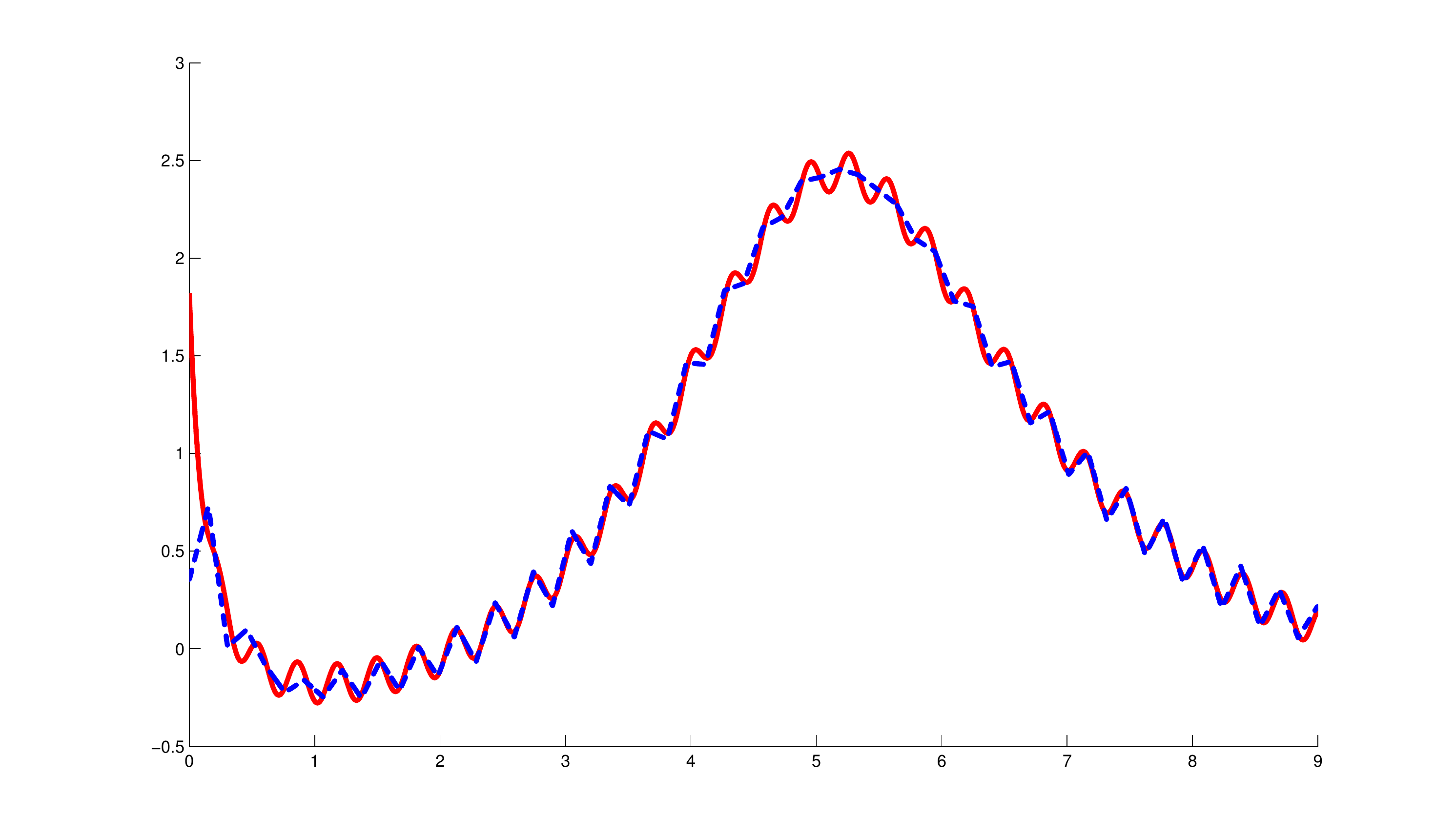}\hspace{-0.9cm}
\includegraphics[width=0.55\textwidth]{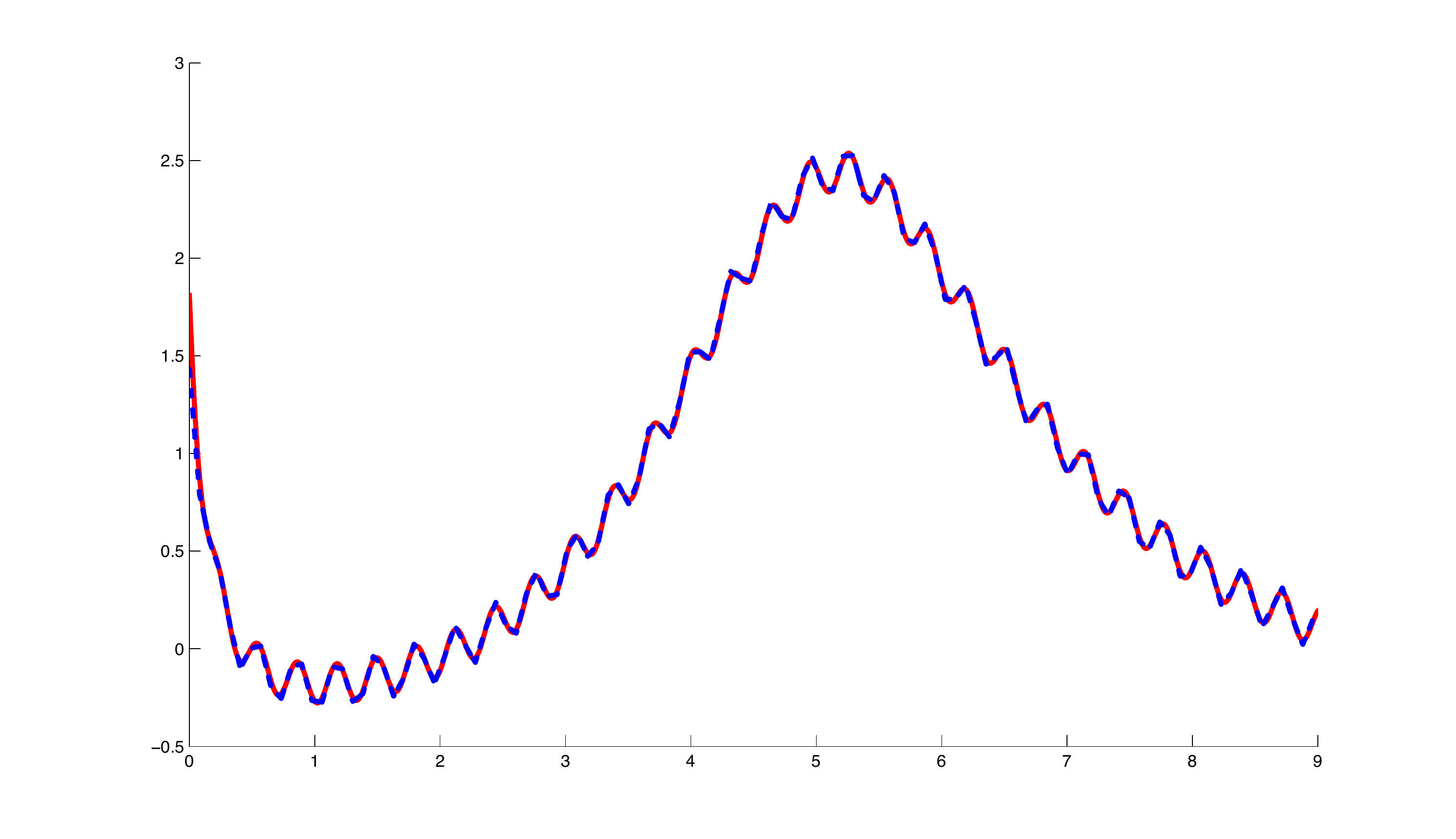}
\end{center}
\caption{Iterative reconstruction of a potential with a singularity at the origin and highly oscillatory behavior. In red: the unknown kernel. In blue: its reconstruction by minimization of $\mathcal{E}^{[a],N}$. From left-top to right-bottom: reconstruction with $N = 10, 20, 40, 80$ agents.}\label{variableN2}
\end{figure}

\subsection{Numerical validation of the coercivity condition}\label{numcoer}

We now turn our attention to the coercivity constant $c_T$ appearing in \eqref{eq-coercive} and thoroughly discussed in Section \ref{sec:coerc}. In Figure \ref{errorN} we see a comparison between the evolution of the value of the error functional $\mathcal{E}^{[a],N}_\Delta(\widehat{a}_N)$ and of the $L_2(\R_+,\rho^N)$-error $\|a-\widehat{a}_N\|^2_{L_2(\R_+,\rho^N)}$ for different values of $N$. 

\begin{figure}[h]
\hspace{-1.2cm}
\begin{minipage}{0.58\textwidth}
\begin{center}
\includegraphics[width=1\textwidth]{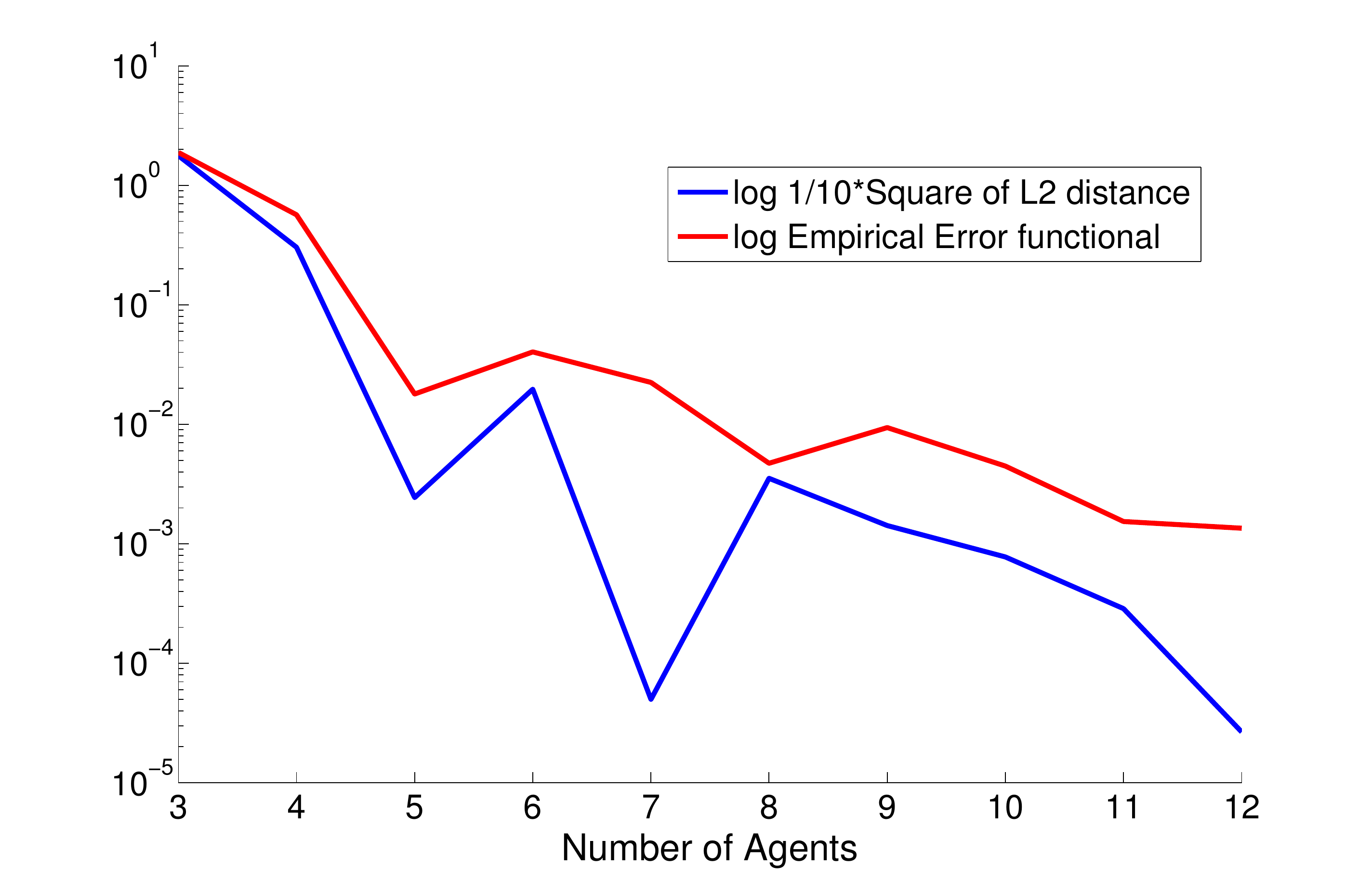}
\end{center}
\label{fig:coerc}
\caption{Plot in logarithmic scale of $\mathcal{E}^{[a],N}(\widehat{a}_N)$ and $\frac{1}{10}\|a-\widehat{a}_N\|^2_{L_2(\R_+,\rho^N)}$ for different values of $N$. In this experiment, we can estimate the constant $c_T$ with the value $\frac{1}{10}$.}\label{errorN}
\end{minipage}
\hspace{0.4cm}
\begin{minipage}{0.55\textwidth}
\begin{center}
\includegraphics[width=1\textwidth]{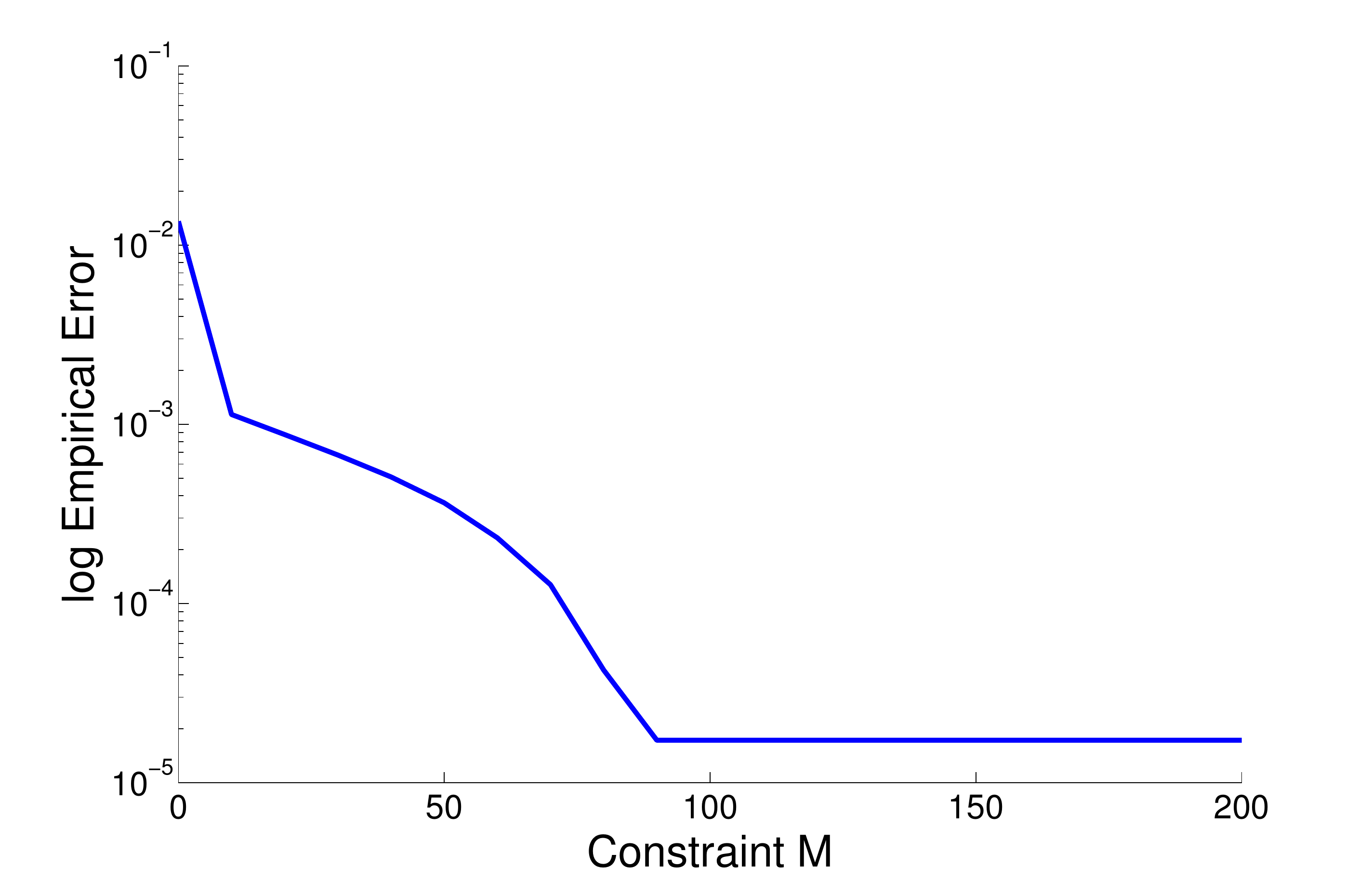}
\end{center}
\label{fig:coerc2}
\caption{Values in logarithmic scale of $\mathcal{E}^{[a],N}_\Delta(\widehat a_{N,M})$ for fixed $N = 50$ for different values of $M \in [0,200]$.}\label{Mconstr}
\end{minipage}
\end{figure}

In this experiment, the potential $a$ to be retrieved is the truncated Lennard-Jones type interaction kernel of Figure \ref{variableN} and the parameters used in the algorithm are reported in Table \ref{tab:fig3}.

\begin{table}[h]
\begin{center}
\begin{tabular}{ |c|c|c|c|c|c| }
\hline
  $d$ & $L$ & $T$ & $M$ & $N$ & $D(N)$ \\
\hline
\hline
  $2$ & $5$ & $0.5$ & $100$ & $[3,4,\ldots,12]$ & $3N-5$ \\
\hline
\end{tabular}
\end{center}
\vspace{-0.5cm}
\caption{Parameter values for Figure \ref{errorN}.} \label{tab:fig3} 
\end{table}

For every value of $N$, we have obtained the minimizer $\widehat{a}_N$ of problem \eqref{problem2} and we have computed the errors $\mathcal{E}^{[a],N}(\widehat{a}_N)$ and $\|a-\widehat{a}_N\|^2_{L_2(\R_+,\rho^N)}$. The $L_2(\R_+,\rho^N)$-error multiplied by a factor $\frac{1}{10}$ lies entirely below the curve of $\mathcal{E}^{[a],N}(\widehat{a}_N)$, which let us empirically estimate the value of $c_T$ around that value (see Figures \ref{fig:coerc} and \ref{fig:coerc2}).
%

\subsection{Tuning the constraint $M$}

Figure \ref{Mconstr1} shows what happens when we modify the value of $M$ in problem \eqref{problem2}. More specifically, we generate $\mu^N_0$ as explained in Section \ref{numfram} once, and we simulate the system starting from $\mu^N_0$ until time $T$. With the data of this single evolution, we solve problem \eqref{problem2} for several values of $M$ and we denote with $\widehat{a}_M \equiv \widehat{a}_{N,M} \equiv \widehat{a}_{N}$ the minimizer obtained with a specific value of $M$. On the left side of Figure \ref{Mconstr1} we show how the reconstruction $\widehat{a}_M$ gets closer and closer to the true potential $a$ (in white) as $M$ increases, while on the right side we illustrate how the original trajectories (again, in white) used for the inverse problem are approximated better and better by those generated with the computed approximation $\widehat{a}_M$, if we let $M$ grow. Table \ref{tab:figM} reports the values of the parameters of these experiments.

\begin{table}[h]
\begin{center}
\begin{tabular}{ |c|c|c|c|c|c|c| }
\hline
 & $d$ & $L$ & $T$ & $M$ & $N$ & $D(N)$ \\
\hline
\hline
 First row & $2$ & $3$ & $1$ & $2.7 \times [10,15,\ldots,40]$ & $20$ & $60$ \\
\hline
 Second row & $2$ & $3$ & $1$ & $1.25 \times [10,15,\ldots,40]$ & $20$ & $150$ \\
\hline
\end{tabular}
\end{center}
\vspace{-0.5cm}
\caption{Parameter values for Figure \ref{Mconstr1}.} \label{tab:figM} 
\end{table}

\begin{figure}[h!]
\begin{center}
\hspace{-0.7cm}\includegraphics[width=0.625\textwidth]{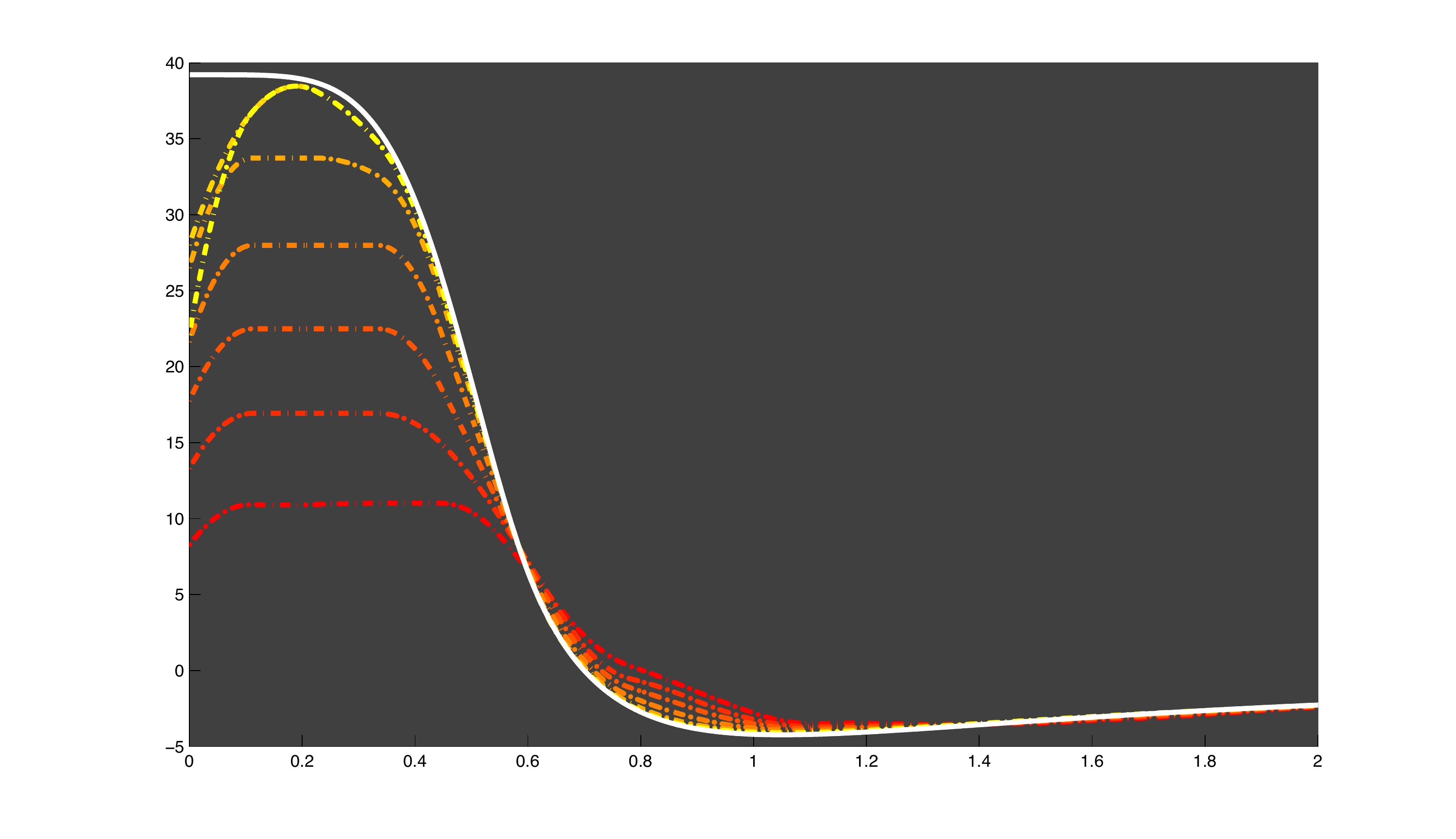}\hspace{-0.9cm}
\includegraphics[width=0.475\textwidth]{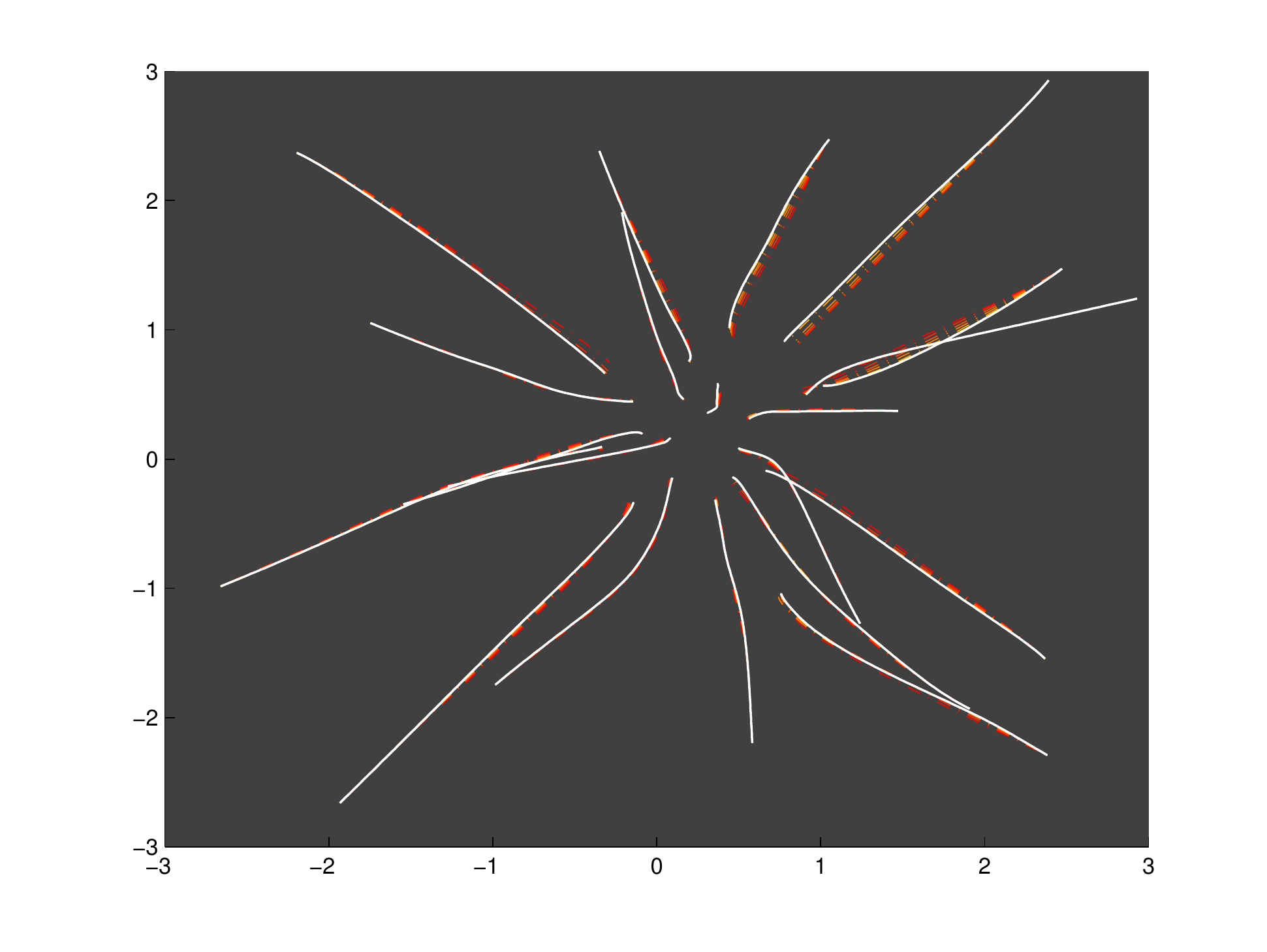}\\
\hspace{-0.7cm}\includegraphics[width=0.645\textwidth]{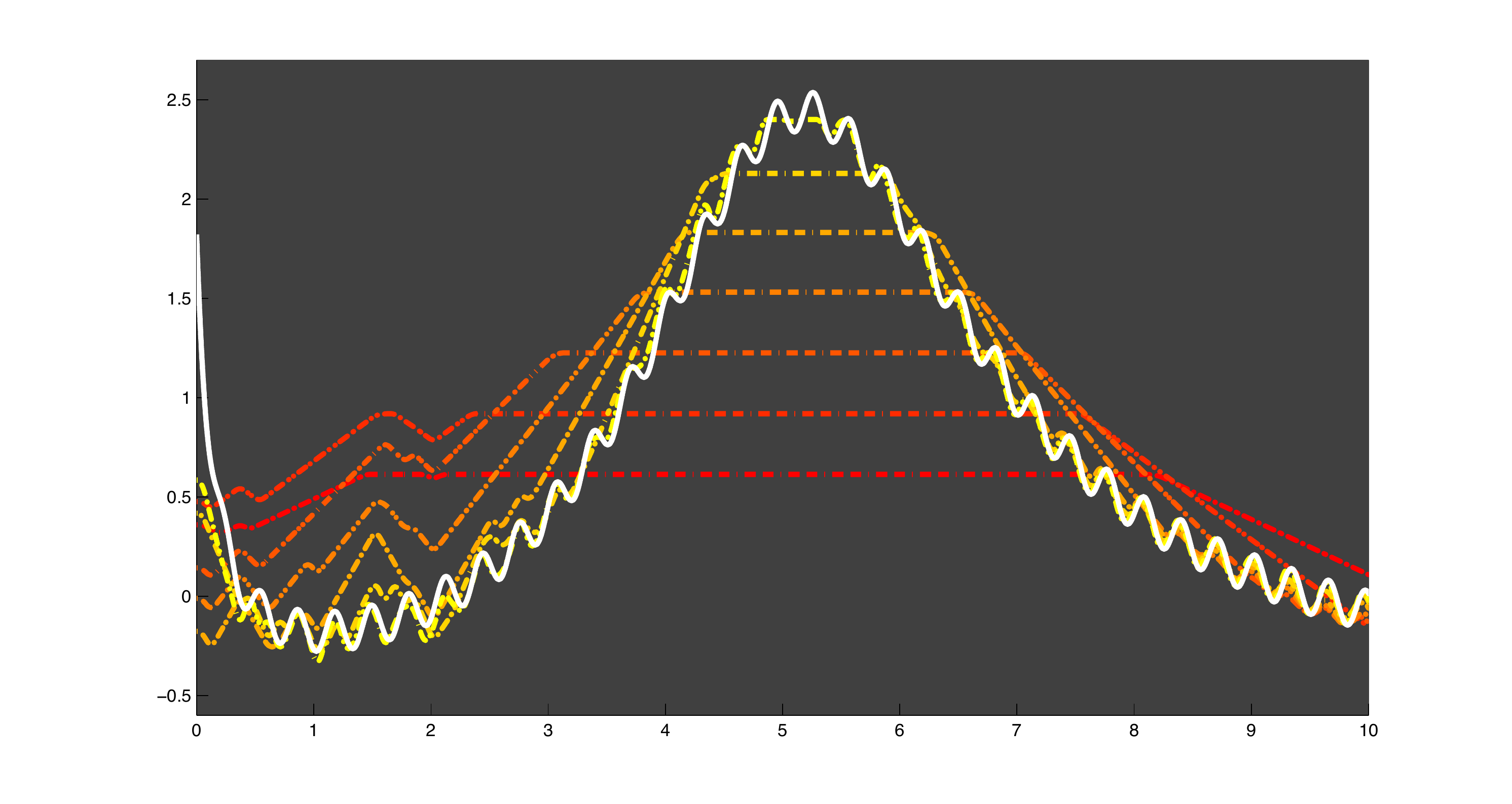}\hspace{-0.9cm}
\includegraphics[width=0.455\textwidth]{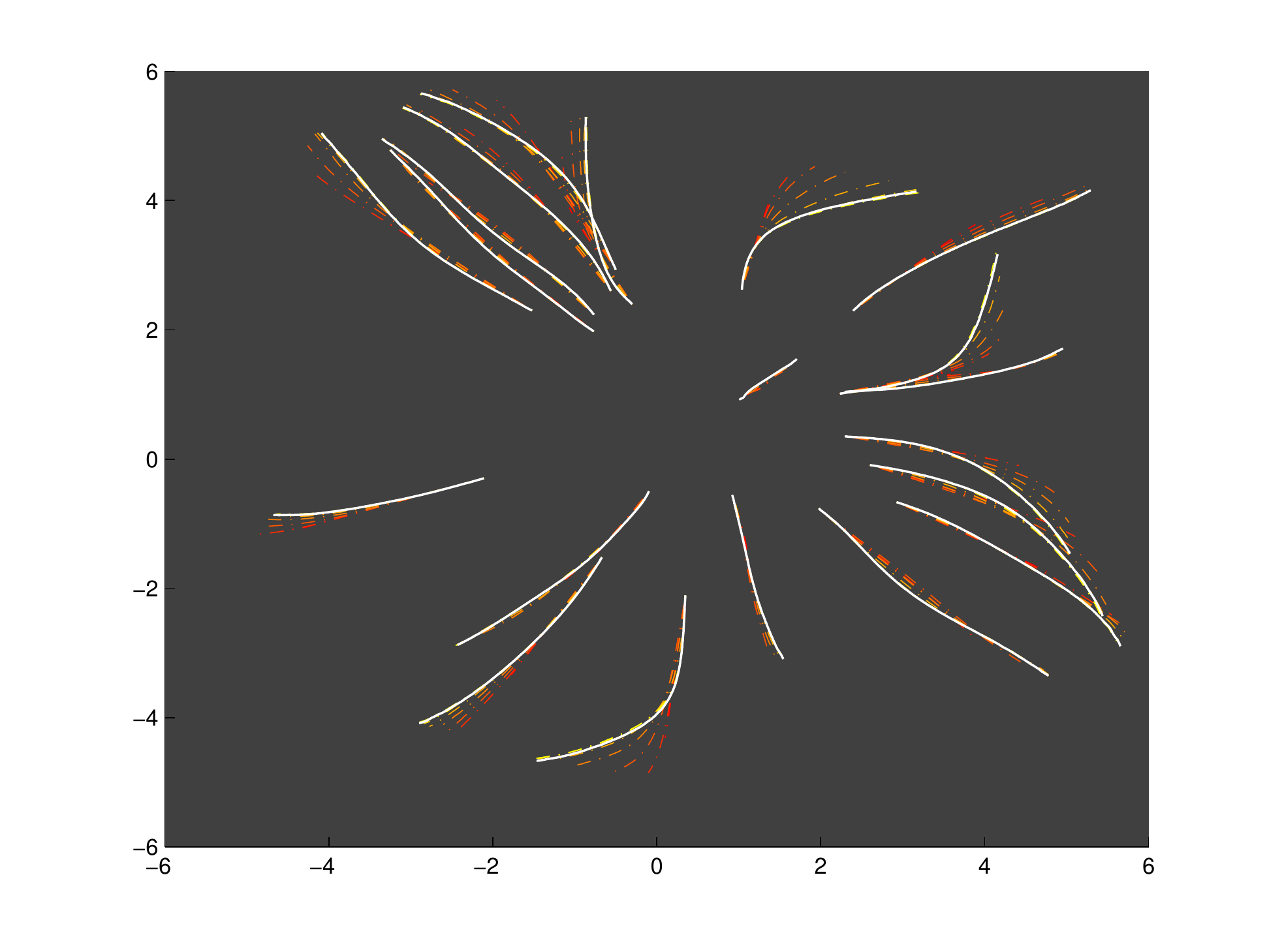}
\end{center}
\caption{Different reconstructions of a potential for different values of $M$. On the left column: the true kernel in white and its reconstructions for different $M$; the brighter the curve, the larger the $M$. On the right column: the true trajectories of the agents in white, the trajectories associated to the reconstructed potentials with the same color.}\label{Mconstr1}
\end{figure}

So far we have no  {\it a priori} criteria to sieve those values of $M$, which enable a successful reconstruction of a potential $a \in X$. However,  the tuning {\it a posteriori} of the parameter $M>0$ turns out to be rather easy. In fact, for $N$ fixed the minimizers $\widehat a_{N,M}$ have the property that the map
$$
 M \mapsto   \mathcal E^{[a],N}(\widehat a_{N,M})
$$
is monotonically decreasing as a function of the constraint parameter $M$ and it becomes constant for $M\geq M^*$, for $M^*>0$ which, as shown empirically, does not depend on $N$. This special value $M^*$ is indeed the ``right" parameter for the $L_\infty$ bound. For such a choice, we show that, if we let $N$ grow, the minimizers $\widehat{a}_N$ approximates better and better the unknown potential $a$. 
Figure \ref{Mconstr} documents precisely this expected behavior.

\begin{figure}[h!]
\begin{center}
\includegraphics[width=0.8\textwidth]{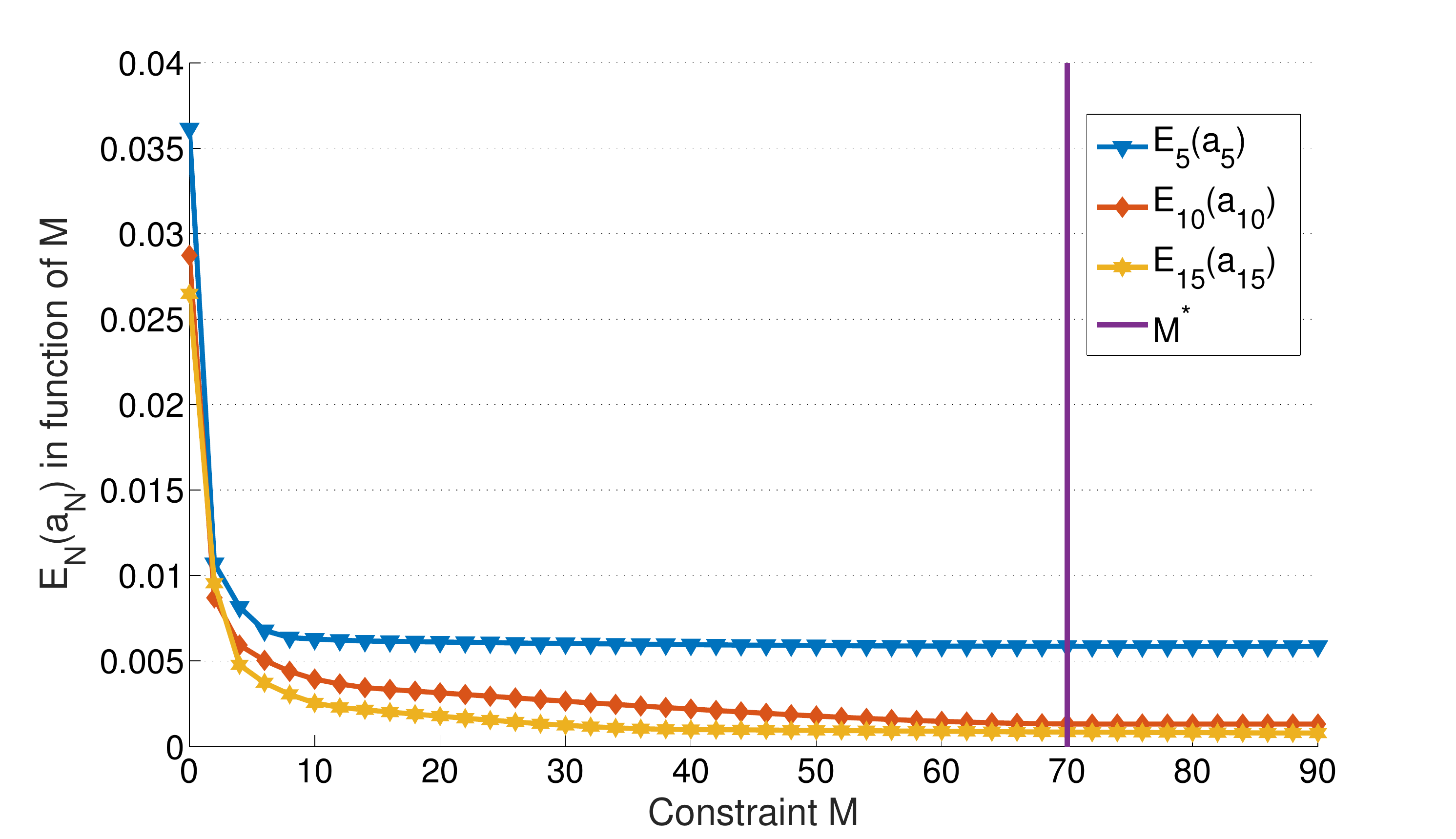}
\end{center}
\caption{Behavior of the error $\mathcal E^{[a],N}(\widehat a_{N,M})$ as a function of the constraint $M$ for different values of $N$.}\label{Mconstr}
\end{figure}

\subsection{Montecarlo-like reconstructions for $N$ fixed}

We mimic now the  mean-field reconstruction strategy, by multiple randomized draw of $N$ particles as initial conditions i.i. distributed according to $\mu_0$ for $N$ fixed and relatively small. Indeed, problem \eqref{problem2} can swiftly become computationally unfeasible when $N$ is moderately large, also because the dimension of the approximating subspaces $V_N$ needs to increase with $N$ too.  We consider, for a fixed $N$, several discrete initial data $(\mu^N_{0,\theta})_{\theta= 1}^{\Theta}$ all independently drawn from the same distribution $\mu_0$ (in our case, the $d$-dimensional cube $[-L,L]^d$). For every $\theta = 1,\ldots,\Theta$, we simulate the system until time $T$ and, with the trajectories we obtained, we solve problem \eqref{problem2}. At the end of this procedure, we have a family of reconstructed potentials $(\widehat{a}_{N,\theta})_{\theta= 1}^{\Theta}$, all approximating the same true kernel $a$. Empirically averaging these potentials, we obtain an approximation
\begin{align*}
\widehat{a}_N(r) = \frac{1}{\Theta} \sum^{\Theta}_{\theta = 1}\widehat{a}_{N,\theta}(r), \quad \text{ for every } r \in [0,R],
\end{align*}
which we claim to be a better approximation to the true kernel $a$ than any single snapshots. To support this claim, we report in Figure \ref{fixedN} the outcome of an experiment whose data can be found in Table \ref{tab:fig5}.

\begin{table}[h]
\begin{center}
\begin{tabular}{ |c|c|c|c|c|c|c| }
\hline
  $d$ & $L$ & $T$ & $M$ & $N$ & $D(N)$ & $\Theta$ \\
\hline
\hline
  $2$ & $2$ & $0.5$ & $1000$ & $50$ & $150$ & 5  \\
\hline
\end{tabular}
\end{center}
\vspace{-0.5cm}
\caption{Parameter values for the experiment of Figure \ref{fixedN}.} \label{tab:fig5} 
\end{table}

\begin{figure}[h!]
\begin{center}
\hspace{-1cm}
\includegraphics[width=0.56\textwidth]{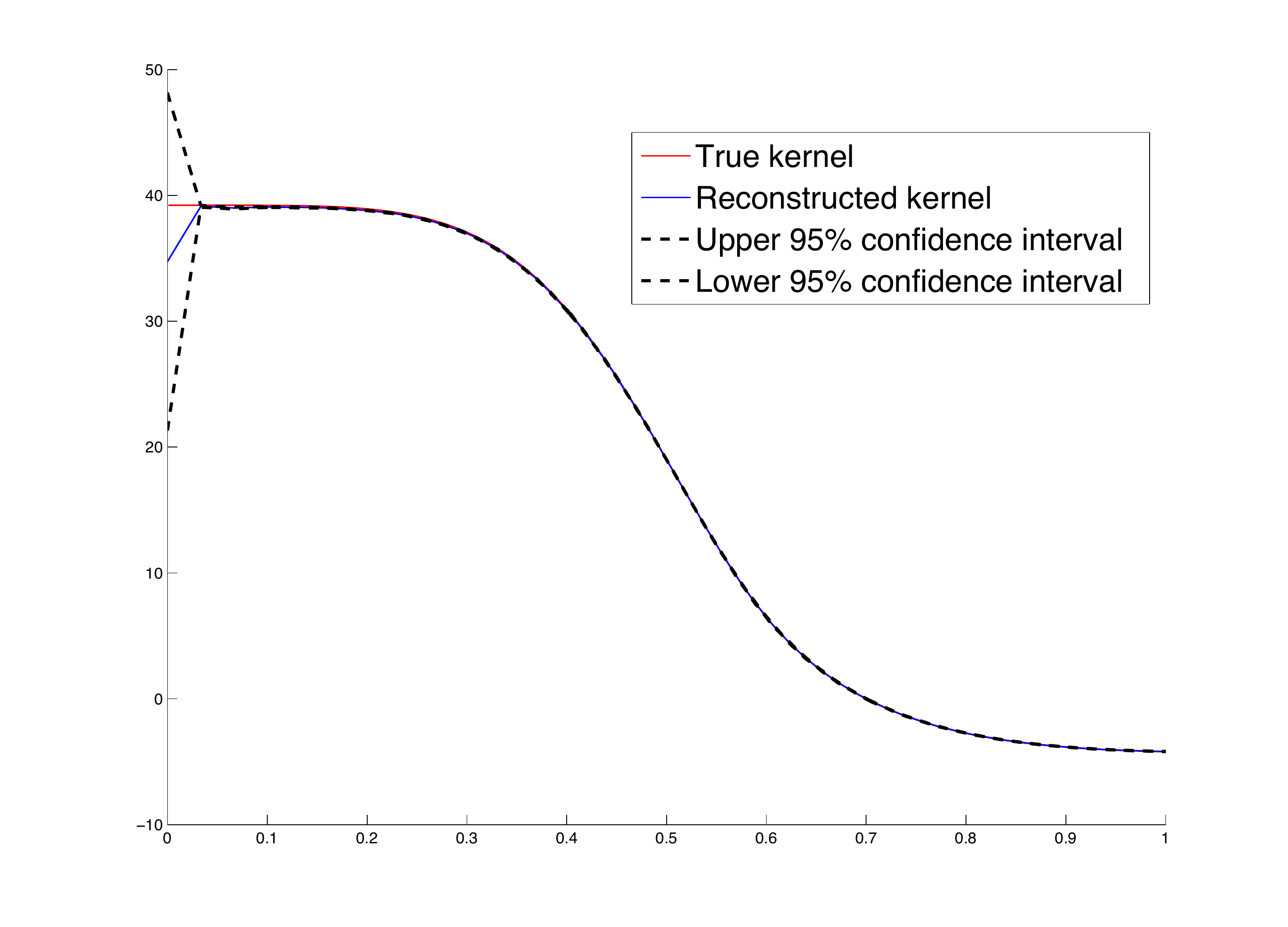}\hspace{-0.9cm}
\includegraphics[width=0.55\textwidth]{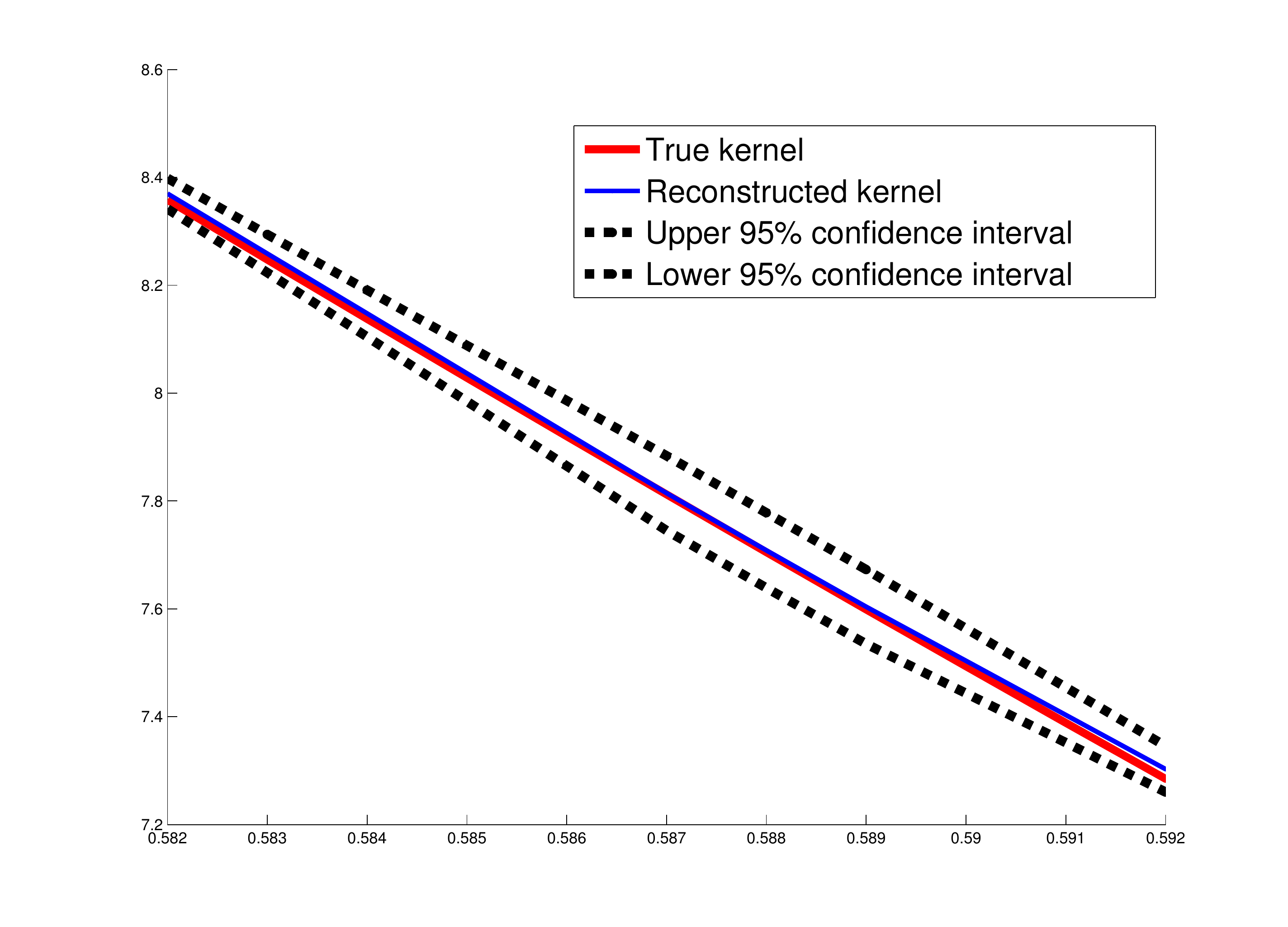}
\end{center}
\caption{Reconstruction of $a$ obtained by averaging 5 solutions of the minimization of $\mathcal{E}^{[a],N}_\Delta$ for $N = 50$. In red: the unknown kernel. In blue: the average of reconstructions. In black: 95\% confidence interval for the parameter estimates returned by the Matlab function \texttt{normfit}. The figure on the right shows a zoom of the left figure.}\label{fixedN}
\end{figure}

\section*{Acknowledgement}

Mattia Bongini, Massimo Fornasier, and Markus Hansen acknowledge the financial support of the ERC-Starting Grant (European Research Council, 306274) “High-Dimensional Sparse Optimal Control” (HDSPCONTR). Mauro Maggioni acknowledges the support of ONR-N00014-12-1-0601 and NSF-ATD/DMS-12-22567. The authors acknowledge the hospitality and the financial support of the University of Bonn and the Hausdorff Center for Mathematics during the Hausdorff Trimester Program ``Mathematics of Signal Processing" for the final preparation of this work.

\section{Appendix}

 Although  similar results on the limit relationship between ODE systems of the type \eqref{eq:discrdyn} and their mean-field equations \eqref{eq:contdyn}
appear in different forms in other papers, see, e.g., \cite{AGS,CanCarRos10,13-Carrillo-Choi-Hauray-MFL,MFOC}, in this Appendix we collect them for our specific setting in a nutshell for the sake of being self-contained and for the convenience of those readers less familiar with these properties of evolutive systems.

\subsection{Standard results on existence and uniqueness for ODE}\label{ap00}

For the reader's convenience and for the sake of a self-contained presentation, we start by briefly recalling some general, well-known results about solutions to Carath{\'e}odory differential equations. We fix a domain $\Omega \subset \R^d$, a Carath{\'e}odory function $g\colon[0,T]\times \Omega \to \R^d$, i.e. the function $g$ is continuous in $y$ and measurable in $t$, and $0<\tau \le T$. A function $y\colon [0,\tau]\to \Omega$ is called a solution of the Carath{\'e}odory differential equation
\begin{equation}\label{cara}
\dot y(t)=g(t, y(t))
\end{equation}
on $[0,\tau]$ if and only if $y$ is absolutely continuous and \eqref{cara} is satisfied a.e.\ in $[0,\tau]$.
The following well-known local existence result holds, see \cite[Chapter 1, Theorem 1]{Fil} .

\begin{theorem}\label{cara-local}
Fix $T > 0$ and $y_0 \in \R^d$. Suppose that there exists a compact subset $\Omega$ of $\R^d$ such that $y_0 \in \textup{int}(\Omega)$ and there exists $m_{\Omega} \in L_1([0,T])$ for which it holds
\begin{align}\label{l1}
|g(t,y)|\le m_{\Omega}(t),
\end{align}
for a.e.\ $t \in [0,T]$ and for all $y \in \Omega$. Then there exists a $\tau > 0$ and a solution $y(t)$ of \eqref{cara} defined on the interval $[0,\tau]$ which satisfies $y(0)=y_0$. 
\end{theorem}

The result can be extended to a global existence as follows.

\begin{theorem}\label{cara-global}
Consider an interval $[0,T]$  on the real line and a  Carath{\'e}odory function $g\colon[0,T]\times \R^d \to \R^d$.
Assume that there exists a constant $C > 0$ such that the function $g$ satisfies the condition
\begin{align}\label{ttz}
|g(t,y)|\le C(1+|y|),
\end{align}
for a.e.\ $t \in [0,T]$ and every $y \in \mathbb R^d$. Then there exists a solution $y(t)$ of \eqref{cara} defined on the whole interval $[0,T]$, which satisfies $y(0)=y_0$. Moreover, for every $t \in [0,T]$, any solution satisfies
\begin{equation}\label{gron}
|y(t)|\le \Big(|y_0|+ Ct\Big) \,e^{Ct}.
\end{equation}
\end{theorem}

\begin{proof}
Set $\rho:= (|y_0|+CT) \,e^{CT}$. Consider now a ball $\Omega \subset \mathbb R^n$ centered at $0$ with radius strictly greater than $\rho$. Existence of a local solution defined on an interval $[0,\tau]$ and taking values in $\Omega$ follows now easily from \eqref{ttz} and Theorem \ref{cara-local}. If \eqref{ttz} holds, any solution of \eqref{cara} with initial datum $y_0$ satisfies
$$
|y(t)|\le |y_0|+ Ct+C\int_0^t |y(s)|\,ds
$$
for every $t \in [0,\tau]$, therefore \eqref{gron} follows from Gronwall's inequality. In particular the graph of a solution $y(t)$ cannot reach the boundary of $[0,T]\times B(0,|y_0|+CTe^{CT})$ unless $\tau=T$, therefore the continuation of the local solution to a global one on $[0,T]$ follows, for instance, from \cite[Chapter 1, Theorem 4]{Fil}.
\end{proof}

A further application of Gronwall's inequality yields the following results on continuous dependence on the initial data.

\begin{proposition}\label{le:uniquecara}
Let $g_1$ and $g_2\colon[0,T]\times \R^n \to \R^n$ be Carath{\'e}odory functions both satisfying \eqref{ttz} for the same  constant $C > 0$. Let $r>0$ and define 
\begin{align*}
\rho_{r, C, T}:=\Big(r+ CT\Big) \,e^{CT}\,.
\end{align*}
Assume in addition that there exists a constant $L > 0$ satisfying
\begin{align*}
|g_1(t, y_1)-g_1(t, y_2)|\le L|y_1-y_2|
\end{align*}
for every $t \in [0, T]$ and every $y_1$, $y_2$ such that $|y_i|\le \rho_{r, C, T}$, $i=1,2$.
Then, if $\dot y_1(t)=g_1(t, y_1(t))$, $\dot y_2(t)=g_2(t, y_2(t))$, $|y_1(0)|\le r$ and $|y_2(0)|\le r$, one has
\begin{equation}\label{gronvalla}
|y_1(t)-y_2(t)|\le e^{Lt}\left(|y_1(0)-y_2(0)|+\int_0^t \|g_1(s, \cdot)-g_2(s, \cdot)\|_{L_\infty(B(0, \rho_{r, C, T}))} \,ds \right)
\end{equation}
for every $t \in [0, T]$.
\end{proposition}
\begin{proof}
We can bound $|y_1(t) - y_2(t)|$ from above as follows:
\begin{align*}
|y_1(t) - y_2(t)| &\leq |y_1(0) - y_2(0)| + \int^t_0 |\dot{y}_1(s) - \dot{y}_2(s)| ds \\
&= |y_1(0) - y_2(0)| \\
& \quad + \int^t_0 |g_1(s, y_1(s)) - g_1(s, y_2(s)) + g_1(s, y_2(s)) - g_2(s, y_2(s))| ds \\
& \leq |y_1(0) - y_2(0)| + \int_0^t \|g_1(s, \cdot)-g_2(s, \cdot)\|_{L_\infty(B(0, \rho_{r, C, T}))} \,ds \\
& \quad  + L \int^t_0|y_1(s) - y_2(s)| ds.
\end{align*}
Since the function $\alpha(t) = |y_1(0) - y_2(0)| + \int_0^t \|g_1(s, \cdot)-g_2(s, \cdot)\|_{L_\infty(B(0, \rho_{r, C, T}))} \,ds$ is increasing, an application of Gronwall's inequality gives \eqref{gronvalla}, as desired.
\end{proof}

\subsection{Technical results for the mean-field limit}\label{ap1}

Let us start this section with some lemmas concerning the growth and the Lipschitz continuity of the right-hand side of  \eqref{eq:discrdyn} .

\begin{lemma}\label{p-estkernel}
Let $a\in X$ and $\mu \in \PP(\R^d)$. Then for all $y \in \R^d$ the following hold:
\begin{align*}
|(\Fun{a} * \mu)(y)| \leq \|a\|_{L_{\infty}(\R_+)}\left( | y | + \int_{\R^d} | x | d\mu(x) \right).
\end{align*}
\end{lemma}
\begin{proof}
It follows directly from $a \in L_{\infty}(\R_+)$.
\end{proof}

\begin{lemma}\label{p-Floclip}
If $a\in X$ then $\Fun{a} \in \Lip_\loc(\R^d)$.
\end{lemma}
\begin{proof}
For any compact set $K \subset \R^d$ and for every $x,y \in K$ it holds
\begin{align*}
|\Fun{a}(x) - \Fun{a}(y)| &= |a(|x|)x - a(|y|)y| \\
&\leq |a(|x|)| |x-y| + |a(|x|) - a(|y|)| |y| \\
&\leq (|a(|x|)| + \Lip_K(a) |y|) |x-y|,
\end{align*}
and since $a \in L_{\infty}(\R_+)$ and $y \in K$, it follows that $\Fun{a}$ is locally Lipschitz with Lipschitz constant depending only on $a$ and $K$.
\end{proof}

\begin{lemma}\label{p-Fmuloclip}
If $a\in X$ and $\mu \in \mathcal{P}_c(\R^d)$ then $\Fun{a}*\mu \in \Lip_{\loc}(\R^d)$.
\end{lemma}
\begin{proof}
For any compact set $K \subset \R^d$ and for every $x,y \in K$ it holds
\begin{align*}
|(\Fun{a}*\mu)(x) - (\Fun{a}*\mu)(y)| &= \left|\int_{\R^d}a(|x-z|)(x-z)d\mu(z) - \int_{\R^d}a(|y-z|)(y-z)d\mu(z)\right| \\
&\leq \int_{\R^d}|a(|x-z|)-a(|y-z|)|x-z|d\mu(z)\\
&\quad+ \int_{\R^d}|a(|y-z|)||x-y|d\mu(z) \\
&\leq \Lip_{\widehat{K}}(a)|x-y| \int_{\R^d}|x-z|d\mu(z) + \|a\|_{L_{\infty}(\R_+)}|x-y| \\
& \leq \left(C\Lip_{\widehat{K}}(a) + \|a\|_{L_{\infty}(\R_+)} \right)|x-y|,
\end{align*}
where $C$ is a constant depending on $K$, and $\widehat{K}$ is a compact set containing both $K$ and $\supp(\mu)$.
\end{proof}

\begin{proposition} Let us fix  $N \in \mathbb N$ and  $a \in X$. Then the system \eqref{eq:discrdyn} admits a unique global solution in $[0,T]$ for every initial datum $x^{N}_0 \in \R^{d \times N}$.
\end{proposition}
\begin{proof}
Let us define  the
 function $g:\R^{d \times N} \rightarrow \R^{d \times N}$ defined for every $x=(x_1, \ldots, x_N)\in \R^{d \times N}$ as
\begin{align*}
g(x_1, \ldots, x_N) = ((\Fun{a}*\mu^N)(x_1),\ldots,(\Fun{a}*\mu^N)(x_N)),
\end{align*}
where $\mu^N$ is the empirical measure given by \eqref{eq:empmeas}. The system \eqref{eq:discrdyn} in the form \eqref{eq:discr1} can be rewritten compactly as
$$
\dot x (t) = g(x(t)).
$$
The function $g$ is clearly a  Carath{\'e}odory function and, by Lemma \ref{p-estkernel}, it clearly satisfies a sublinear growth condition of the type \eqref{ttz}. Moreover it is also locally Lipschitz continuous: indeed, for any $x_1, \ldots, x_N, y_1, \ldots, y_N \in K$ compact subset of $\R^d$, denoting with $\nu^N$ the empirical measure given by $y_1, \ldots, y_N$, it simply suffices to write
\begin{align*}
|g(x_1, \ldots, x_N) - g(y_1,\ldots,y_N)| &\leq \sum^N_{i = 1} |(\Fun{a}*\mu^N)(x_i) - (\Fun{a}*\nu^N)(y_i)| \\
&\leq \sum^N_{i = 1} \Bigg( |(\Fun{a}*\mu^N)(x_i) - (\Fun{a}*\mu^N)(y_i)| \\
&\quad \quad \quad \quad +|(\Fun{a}*\mu^N)(y_i) - (\Fun{a}*\nu^N)(y_i)| \Bigg).
\end{align*}
Applying Lemma \ref{p-Fmuloclip} to the first term and performing similar calculations to the ones in the proof of Lemma \ref{p-Floclip} on the second one, gives the desired result. We conclude the existence of a unique global solution by an application of Theorem \ref{cara-global} and its uniqueness follows from Lemma \ref{le:uniquecara}.
\end{proof}

The following preliminary result tells us that solutions to system \eqref{eq:discrdyn} are also solutions to the equation \eqref{eq:contdyn}, whenever conveniently rewritten.

\begin{proposition}\label{p-rewritten}
Let $N \in \N$ be given and $a \in X$. Let $(x^N_1, \ldots, x^N_N):[0,T] \rightarrow \R^{d\times N}$ be the solution of \eqref{eq:discrdyn} with initial datum $x^{N}_0 \in \R^{d \times N}$. Then the empirical measure $\mu^N:[0,T] \rightarrow \PP(\R^d)$ defined as in \eqref{eq:empmeas} is a solution of \eqref{eq:contdyn} with initial datum $\mu_{0}= \mu^N(0) \in \PC(\R^d)$.
\end{proposition}
\begin{proof}
It can be  proved by testing the equation  \eqref{eq:contdyn}  against a continuously differentiable function, arguing exactly as in \cite[Lemma 4.3]{MFOC}.
\end{proof}

\subsection{Existence and uniqueness of solutions  for  \eqref{eq:contdyn} }\label{ap3}

Variants of the following result are \cite[Lemma 6.7]{MFOC} and \cite[Lemma 4.7]{CanCarRos10}

\begin{lemma}\label{p-lipkernel}
Let $a \in X$ and let $\mu:[0,T] \rightarrow \mathcal{P}_c(\R^d)$ and $\nu: [0,T] \to \mathcal{P}_c(\R^d)$ be two continuous maps with respect to $\W_1$ satisfying
\begin{align}\label{eq:bsupp}
\supp(\mu(t)) \cup \supp(\nu(t)) \subseteq B(0,R),
\end{align}
for every $t \in [0,T]$, for some $R > 0$. Then for every $r > 0$ there exists a constant $L_{a,r,R}$ such that
\begin{align}\label{eq:inftynormW1}
\|\Fun{a} * \mu(t) - \Fun{a} * \nu(t)\|_{L_{\infty}(B(0,r))} \leq L_{a,r,R} \W_1(\mu(t),\nu(t))
\end{align}
for every $t \in [0,T]$.
\end{lemma}
\begin{proof}
Fix $t \in [0,T]$ and take $\pi \in \Gamma_o(\mu(t),\nu(t))$. Since the marginals of $\pi$ are by definition $\mu(t)$ and $\nu(t)$, it follows
\begin{align*}
\Fun{a} * \mu(t)(x) - \Fun{a} * \nu(t)(x) &= \int_{B(0,R)} \Fun{a}(x-y) d\mu(t)(y) - \int_{B(0,R)} \Fun{a}(x-z) d\nu(t)(z)  \\
&= \int_{B(0,R)^2} \left(\Fun{a}(x-y) - \Fun{a}(x-z)\right) d\pi(y,z)
\end{align*}
By using Lemma \ref{p-Floclip} and the hypothesis \eqref{eq:bsupp}, we have
\begin{align*}
\|\Fun{a} * \mu(t) - \Fun{a} * \nu(t)\|_{L_{\infty}(B(0,r))} &\leq \esssup_{x \in B(0,r)} \int_{B(0,R)^2} \left|\Fun{a}(x-y) - \Fun{a}(x-z)\right| d\pi(y,z) \\
&\leq \Lip_{B(0,R+r)}(\Fun{a}) \int_{B(0,R)^2} |y - z| d\pi(y,z) \\
&= \Lip_{B(0,R+r)}(\Fun{a}) \W_1(\mu(t),\nu(t)),
\end{align*}
hence \eqref{eq:inftynormW1} holds with $L_{a,r,R} = \Lip_{B(0,R+r)}(\Fun{a})$.
\end{proof}

We show now the proof of Proposition \ref{pr:exist} which states the existence of solutions for \eqref{eq:contdyn} .

\begin{proof}[Proof of Proposition \ref{pr:exist}]

Let us define the quantity $\mathcal X_N(t) := \max_{i=1,\dots,N} |x_i^N(t)|$. By integration of \eqref{eq:discr1} we obtain
\begin{eqnarray*}
 |x_i^N(t)| &\leq& |x^N_{0,i}| + \int_0^t (\Fun{a} * \mu^N(s))(x_i^N)| ds \\
&\leq&  |x^N_{0,i}| + \int_0^t\frac{1}{N} \sum_{j=1}^N |a(|x_i-x_j|)||x_j-x_i| ds \\
&\leq&  |x^N_{0,i}| + \|a\|_{L_{\infty}(\R_+)} \int_0^t\frac{1}{N} \sum_{j=1}^N( |x_j| + | x_i| )ds,
\end{eqnarray*}
implying
$$
\mathcal X_N(t) \leq \mathcal X_N(0) + 2  \|a\|_{L_{\infty}(\R_+)} \int_0^t \mathcal X_N(s) ds. 
$$
Hence, Gronwall's Lemma and the hypothesis $x^{N}_{0,i} \in \supp(\mu_0)$ for every $N \in \N$ and $i = 1, \ldots, N$, imply that
\begin{align*}
\mathcal X_N(t) \leq \mathcal X_N(0) e^{2 \|a\|_{L_{\infty}(\R_+)} t} \leq C_0 e^{2 \|a\|_{L_{\infty}(\R_+)} t} \text{ for a.e. } t \in [0,T],
\end{align*}
for some uniform constant $C_0$ depending only on $\mu_0$. Therefore, the support of the empirical measure $\mu^N(\cdot)$ is bounded uniformly in $N$ in a ball $B(0,R) \subset \R^d$, where
\begin{align}\label{Rest}
R =  C_0 e^{2 \|a\|_{L_{\infty}(\R_+)} T}.
\end{align}
Now, notice that from \eqref{dualwass} it follows
$$
\mathcal W_1(\mu^N(t), \mu^N(s)) \leq \frac{1}{N} \sum_{i=1}^N | x_i^N(t) - x_i^N(s)| ,
$$
and the local Liptschitz contiunuity of $\mu^N(t)$ follows from the one of $x_i^N(t)$: indeed  $|x^N_i(t)| \leq R$ for a.e. $t \in [0,T]$, for all $N \in N$ and $i = 1, \ldots, N$, and Lemma \ref{p-estkernel} yields
\begin{align*}
|\dot{x}^N_i(t)| &= |(\Fun{a}*\mu^N(t))(x^N_i(t))| \\
&\leq \|a\|_{L_{\infty}(\R_+)} \left( |x^N_i(t)| + \frac{1}{N}\sum^N_{j = 1}|x^N_j(t)|\right) \\
&\leq 2R\|a\|_{L_{\infty}(\R_+)}.
\end{align*}
Hence, the sequence $(\mu^N)_{N \in \N} \subset \mathcal{C}^0([0,T],\mathcal{P}_1(B(0,R)))$ is equicontinuous, because equi-Lipschitz continuous, and equibounded in the complete metric space $(\mathcal{P}_1(B(0,R)),\W_1)$.
Therefore, we can apply the Ascoli-Arzel\'{a} Theorem for functions with values in a metric space (see for instance, \cite[Chapter 7, Theorem 18]{KelleyTop}) to infer the existence of a subsequence $(\mu^{N_k})_{k \in \N}$ of $(\mu^N)_{N \in \N}$ such that
\begin{align}\label{eq:unifconv}
\lim_{k \rightarrow \infty}\W_1(\mu^{N_k}(t),\mu(t)) = 0 \quad \text{ uniformly for a.e. } t \in [0,T],
\end{align}
for some $\mu \in \mathcal{C}^0([0,T],\mathcal{P}_1(B(0,R)))$ with Lipschitz constant bounded by $2R\|a\|_{L_{\infty}(\R_+)}$. The hypothesis $\lim_{N\rightarrow\infty}\W_1(\mu^N_0,\mu_0) = 0$ now obviously implies $\mu(0) = \mu_0$. In particular it holds
\begin{equation}\label{initialdatum}
\lim_{k\to \infty} \langle \varphi, \mu^N(t) - \mu^N(0) \rangle  =  \langle \varphi, \mu(t) - \mu_0 \rangle
\end{equation}
for all $\varphi \in \mathcal{C}^1_c(\R^d;\R)$.

We are now left with verifying that this curve $\mu$ is a solution of \eqref{eq:contdyn}. For all $t \in [0,T]$ and for all $\varphi \in \mathcal{C}^1_c(\R^d;\R)$, it holds
\begin{align*}
\frac{d}{dt}\langle \varphi, \mu^N(t) \rangle = \frac{1}{N}\frac{d}{dt} \sum^N_{i = 1} \varphi(x^N_i(t)) = \frac{1}{N} \sum^N_{i = 1} \nabla\varphi(x^N_i(t)) \cdot \dot{x}_i^N(t).
\end{align*}
By directly applying the substitution $\dot{x}_i^N(t) = (\Fun{a}*\mu^N(t))(x^N_i(t))$, we have
\begin{align*}
\langle \varphi, \mu^N(t) - \mu^N(0) \rangle = \int^t_0 \left[ \int_{\R^d}\nabla \varphi(x) \cdot (\Fun{a}*\mu^N(s))(x) d\mu^N(s)(x) \right] ds.
\end{align*}
By Lemma \ref{p-lipkernel}, the inequality \eqref{eq:inftynormW1}, and the compact support of $\varphi \in \mathcal{C}^1_c(\R^d;\R)$, it follows
\begin{align*}
\lim_{N \rightarrow \infty} \|\nabla\varphi \cdot (\Fun{a}*\mu^N(t) - \Fun{a}*\mu(t))\|_{L_{\infty}(\R^d)} = 0 \quad \text{ uniformly for a.e. } t \in [0,T].
\end{align*}
If we denote with $\mathcal L_1\llcorner_{[0,t]}$ the Lebesgue measure on the time interval $[0,t]$, since the product measures $\frac{1}{t} \mu^{N}(s) \times \mathcal L_1\llcorner_{[0,t]}$ converge in $\mathcal P_1([0,t] \times \mathbb R^{d})$ to $\frac{1}{t} \mu(s) \times \mathcal L_1\llcorner_{[0,t]}$, we finally get from the dominated convergence theorem that
\begin{align}
\lim_{N \to \infty} \int_0^{t} \int_{\mathbb R^{d}} \nabla \phi(x) \cdot (\Fun{a}*&\mu^N(s))(x) d\mu^N(s)(x) ds  \nonumber \\
&=  \int_0^{t} \int_{\mathbb R^{d}} \nabla \phi(x) \cdot (\Fun{a}*\mu(s))(x) d \mu(s)(x) ds, \label{limitmf}.
\end{align}
The statement now follows from combination of \eqref{initialdatum} and \eqref{limitmf}.  
\end{proof}

\begin{proposition}\label{exmono}
Fix $T > 0$, $a \in X$, $\mu_0 \in \mathcal{P}_c(\R^d)$, $\xi_0 \in \R^d$ and $R > 0$. 
For every map $\mu:[0,T] \rightarrow \PP(\R^d)$ which is continuous with respect to $\W_1$ such that
\begin{align*}
\supp(\mu(t)) \subseteq B(0,R) \quad \text{ for every } t \in [0,T],
\end{align*}
there exists a unique solution of system \eqref{eq:transpdyn} with initial value $\xi_0$ defined on the whole interval $[0,T]$.
\end{proposition}
\begin{proof}
The statement follows again by a proper combination of Lemma \ref{p-estkernel}  and  Lemma \ref{p-Fmuloclip} with Theorem \ref{cara-global}
for the existence, and the uniqueness similarly follows from Proposition \ref{le:uniquecara}.

\end{proof}


The following Lemma and \eqref{gronvalla} are the main ingredients of the proof of Theorem \ref{uniq} on continuous dependance on initial data and uniqueness
of solutions for \eqref{eq:contdyn}.

\begin{lemma}\label{primstim}
Let $\mathcal{T}_1$ and $\mathcal{T}_2 \colon \R^n \to \R^n$ be two bounded Borel measurable functions. Then, for every $\mu \in \PP(\R^n)$ one has
\begin{align*}
\W_1((\mathcal{T}_1)_{\#}\mu, (\mathcal{T}_2)_{\#} \mu) \le \|\mathcal{T}_1-\mathcal{T}_2\|_{L_\infty({\rm supp}\,\mu)}.
\end{align*}
If in addition $\mathcal{T}_1$ is locally Lipschitz continuous, and $\mu$, $\nu \in \PP(\R^n)$ are both compactly supported on a ball $B(0,r)$ of $\R^n$ for $r>0$, then
\begin{align*}
\W_1((\mathcal{T}_1)_{\#} \mu, (\mathcal{T}_1)_{\#} \nu) \le \Lip_{B(0,r)}(E_1) \W_1(\mu, \nu).
\end{align*}
\end{lemma}

\begin{proof}
See \cite[Lemma 3.11]{CanCarRos10} and \cite[Lemma 3.13]{CanCarRos10}.
\end{proof}

We can now prove Theorem \ref{uniq}.

\begin{proof}[Proof of Theorem \ref{uniq}]
Let  ${\mathcal T}^\mu_t$ and ${\mathcal T}^\nu_t$ be the flow maps associated to system \eqref{eq:transpdyn} with measure $\mu$ and $\nu$, respectively.
By \eqref{eq:fixedpoint}, the triangle inequality, Lemma \ref{p-lipkernel}, Lemma \ref{primstim} and \eqref{eq:liptrans} we have for every $t \in [0,T]$
\begin{align}
\begin{split}\label{start}
\W_1(\mu(t), \nu(t))&=\W_1(({\mathcal T}^\mu_t)_{\#} \mu_0, ({\mathcal T}^\nu_t)_{\#} \nu_0)  \\
&\le \W_1(({\mathcal T}^\mu_t)_{\#} \mu_0, ({\mathcal T}^\mu_t)_{\#} \nu_0) + \W_1(({\mathcal T}^\mu_t)_{\#} \nu_0, ({\mathcal T}^\nu_t)_{\#} \nu_0)\\
&\le e^{T \, \Lip_{B(0,R)}(\Fun{a})} \W_1(\mu_0, \nu_0)+\|{\mathcal T}^\mu_t-{\mathcal T}^\nu_t\|_{L_\infty(B(0,R))}.
\end{split}
\end{align}

Using \eqref{gronvalla} with $y_1(0)= y_2(0)$ we get
\begin{equation}\label{stima2}
\|{\mathcal T}^\mu_t-{\mathcal T}^\nu_t\|_{L_\infty(B(0,r))}\le e^{t \, \Lip_{B(0,R)}(\Fun{a})}\int_0^t \|\Fun{a}* \mu(s)-\Fun{a}* \nu(s)\|_{L_\infty(B(0,R))}\,ds.
\end{equation}

Combining \eqref{start} and \eqref{stima2} with Lemma \ref{p-lipkernel}, we have
$$
\W_1(\mu(t), \nu(t))\le e^{T \, \Lip_{B(0,R)}(\Fun{a})} \left(\W_1(\mu^0, \nu_0)+ L_{a,R,R}\int_0^t \W_1(\mu(s), \nu(s)) \,ds\right)
$$
for every $t \in [0, T]$, where $L_{a,R,R}$ is the constant from Lemma \ref{p-lipkernel}. Gronwall's inequality now gives
$$
\W_1(\mu(t), \nu(t))\le e^{T \, \Lip_{B(0,R)}(\Fun{a}) + L_{a,R,R}} \W_1(\mu^0, \nu_0),
$$
which is exactly \eqref{stab} with $\overline{C}= e^{T \, \Lip_{B(0,R)}(\Fun{a}) + L_{a,R,R}}$.

Consider now two solutions of \eqref{eq:contdyn} with the same initial datum $\mu_0$.  By definition they both satisfy \eqref{supptot} for some $R>0 $ and \eqref{stab} guarantees they both describe the same trajectory in $\PP(\R^d)$. This concludes the proof.
\end{proof}

\bibliographystyle{abbrv}
\bibliography{biblio}	
\addcontentsline{toc}{section}{References}

\end{document}